\documentclass[12pt]{amsart} 
\usepackage{amsfonts,amsmath,amsthm,amsfonts,amscd,amssymb,amsmath,latexsym,multicol,euscript}
\usepackage[normalem]{ulem}
\usepackage{epsfig,graphicx,flafter,caption}
\usepackage{enumerate}
\usepackage[active]{srcltx}
\usepackage{hyperref,url}
\usepackage{color}
\usepackage{asymptote,asyfig}
\usepackage{media9}
\usepackage[small,nohug,UglyObsolete]{diagrams}
\diagramstyle[labelstyle=\scriptstyle]
\diagramstyle[noPostScript]
\usepackage{algorithm}
\usepackage[noend]{algpseudocode}
\usepackage{xr}
\usepackage{polynom}
\usepackage{csquotes}
\usepackage{epigraph}

\makeatletter

\def\jobis#1{FF\fi
  \def\preedicate{#1}%
  \edef\preedicate{\expandafter\strip@prefix\meaning\preedicate}%
  \edef\job{\jobname}%
  \ifx\job\preedicate
}

\makeatother

\if\jobis{proposal}%
 \def\try{subsection}%
\else
  \def\try{section}%
\fi

\theoremstyle{plain}
\newtheorem{theorem}{Theorem}[\try]
\newtheorem{corollary}[theorem]{Corollary}
\newtheorem{lemma}[theorem]{Lemma}
\newtheorem{claim}[theorem]{Claim}

\newtheorem{proposition}[theorem]{Proposition}
\newtheorem{definition-lemma}[theorem]{Definition-Lemma}
\newtheorem{definition-proposition}[theorem]{Definition-Proposition}
\newtheorem{definition-theorem}[theorem]{Definition-Theorem}

\newtheorem{definition}[theorem]{Definition}
\newtheorem{remark}[theorem]{Remark}

\def\lfomitlist#1.#2.#3.#4.{{#1}_0,{#1}_1 #2 \dots #2\hat{{#1}_{#4}} #2\dots #2 {#1}_{#3}}

\def\alist#1.#2.#3.{{#1}_1 #2 {#1}_2 #2\dots #2 {#1}_{#3}}
\def\zlist#1.#2.#3.{#1_0 #2 #1_1 #2\dots #2 #1_{#3}}
\def\ltomitlist#1.#2.#3.{{#1}_0,{#1}_1 #2 \dots #2\hat {{#1}_i} #2\dots #2 {#1}_{#3}}
\def\lomitlist#1.#2.#3.{{#1}_0 #2 {#1}_1 #2 \dots #2 \hat {{#1}_i} #2 \dots #2 {#1}_{#3}}
\def\lmap#1.#2.#3.{#1 \overset{#2}{\longrightarrow} #3}
\def\mes#1.#2.#3.{#1 \longrightarrow #2 \longrightarrow #3}
\def\ses#1.#2.#3.{0\longrightarrow #1 \longrightarrow #2 \longrightarrow #3 \longrightarrow 0}
\def\les#1.#2.#3.{0\longrightarrow #1 \longrightarrow #2 \longrightarrow #3}
\def\res#1.#2.#3.{#1 \longrightarrow #2 \longrightarrow #3\longrightarrow 0}
\def\Hi#1.#2.#3.{\text {Hilb}^{#1}_{#2}(#3)}
\def\ten#1.#2.#3.{#1\underset {#2}{\otimes} #3}

\def\mderiv#1.#2.#3.{\frac {d^{#3} #1}{d #2^{#3}}}
\def\mfderiv#1.#2.#3.{\frac {\partial^{#3} #1}{\partial #2}}
\def\ggr#1.#2.#3.{\mathbb{G}_{#1}(#2,#3)}

\def\llist#1.#2.{{#1}_1,{#1}_2,\dots,{#1}_{#2}}
\def\ulist#1.#2.{{#1}^1,{#1}^2,\dots,{#1}^{#2}}
\def\lomitlist#1.#2.{{#1}_1,{#1}_2,\dots,\hat {{#1}_i}, \dots, {#1}_{#2}}
\def\lomitlistz#1.#2.{{#1}_0,{#1}_1,\dots,\hat {{#1}_i}, \dots, {#1}_{#2}}
\def\loc#1.#2.{\Cal O_{#1,#2}}
\def\fderiv#1.#2.{\frac {\partial #1}{\partial #2}}
\def\deriv#1.#2.{\frac {d #1}{d #2}}
\def\map#1.#2.{#1 \longrightarrow #2}
\def\rmap#1.#2.{#1 \dasharrow #2}
\def\emb#1.#2.{#1 \hookrightarrow #2}
\def\non#1.#2.{\text {Spec }#1[\epsilon]/(\epsilon)^{#2}}
\def\Hi#1.#2.{\text {Hilb}^{#1}(#2)}
\def\sym#1.#2.{\operatorname {Sym}^{#1}(#2)}
\def\Hb#1.#2.{\text {Hilb}_{#1}(#2)}
\def\Hm#1.#2.{\Hom_{#1}(#2)}
\def\prd#1.#2.{{#1}_1\cdot {#1}_2\cdots {#1}_{#2}}
\def\Bl #1.#2.{\operatorname {Bl}_{#1}#2}
\def\pl #1.#2.{#1^{\otimes #2}}
\def\mgn#1.#2.{\overline {M}_{#1,#2}}
\def\ialist#1.#2.{{#1}_1 #2 {#1}_2 #2 {#1}_3 #2\dots}
\def\pair#1.#2.{\langle #1, #2\rangle}
\def\gproj#1.#2.{\mathbb{P}_{#1}(#2)}
\def\gpr #1.#2.{\mathbb{P}^{#1}_{#2}}
\def\gaf #1.#2.{\mathbb{A}^{#1}_{#2}}
\def\vandermonde#1.#2.{\left|
\begin{matrix}
1 & 1 & 1 & \dots & 1\\
{#1}_1 & {#1}_2 & {#1}_3 & \dots & {#1}_{#2}\\
{#1}_1^2 & {#1}_2^2 & {#1}_3^2 & \dots & {#1}_{#2}^2\\
\vdots & \vdots & \vdots & \ddots & \vdots\\
{#1}_1^{#2-1} & {#1}_2^{#2-1} & {#1}_2^{#2-1} & \dots & {#1}_{#2}^{#2-1}\\
\end{matrix}
\right|
}
\def\vandermondet#1.#2.{\left|
\begin{matrix}
1 & {#1}_1   & {#1}_1^2 & \dots & {#1}_1^{#2-1}\\
1 & {#1}_2   & {#1}_2^2 & \dots & {#1}_2^{#2-1}\\
1 & {#1}_3   & {#1}_3^2 & \dots & {#1}_3^{#2-1}\\
\vdots & \vdots & \vdots & \ddots & \vdots\\
1 & {#1}_{#2}& {#1}_{#2}^2 & \dots & {#1}_{#2}^{#2-1}\\
\end{matrix}
\right|
}
\def\gr#1.#2.{\mathbb{G}(#1,#2)}
\def\bdd#1.#2.{{#1}_{\rfdown #2.}}

\def\ideal#1.{I_{#1}}
\def\ring#1.{\mathcal {O}_{#1}}
\def\fring#1.{\hat{\mathcal {O}}_{#1}}
\def\aring#1.{{\mathcal {O}}_{#1}^{\text{an}}}
\def\proj#1.{\mathbb {P}(#1)}
\def\pr #1.{\mathbb {P}^{#1}}
\def\dpr #1.{\hat{\mathbb {P}}^{#1}}
\def\af #1.{\mathbb{A}^{#1}}
\def\Hz #1.{\mathbb{F}_{#1}}
\def\Hbz #1.{\overline{\mathbb {F}}_{#1}}
\def\fb#1.{\underset {#1} {\times}}
\def\rest#1.{\underset {\ \ring #1.} \to \otimes}
\def\au#1.{\operatorname {Aut}\,(#1)}
\def\deg#1.{\operatorname {deg } (#1)}
\def\pic#1.{\operatorname {Pic}\,(#1)}
\def\pico#1.{\operatorname{Pic}^0(#1)}
\def\picg#1.{\operatorname {Pic}^G(#1)}
\def\n1#1.{\operatorname{N^1}(#1)}
\def\ner#1.{\operatorname{NS}(#1)}
\def\rdown#1.{\llcorner#1\lrcorner}
\def\rfdown#1.{\lfloor{#1}\rfloor}
\def\rup#1.{\ulcorner{#1}\urcorner}
\def\rfup#1.{\lceil{#1}\rceil}
\def\sship#1.{\langle{#1}\rangle}

\def\bp#1.{#1^{{}\leq 1}}
\def\rcup#1.{\lceil{#1}\rceil}
\def\cone#1.{\operatorname {NE}(#1)}
\def\mone#1.{\operatorname {NM}(#1)}
\def\none#1.{\operatorname {NF}(#1)}
\def\ccone#1.{\overline{\operatorname {NE}}(#1)}
\def\cmone#1.{\overline{\operatorname {NM}}(#1)}
\def\cnone#1.{\overline{\operatorname {NF}}(#1)}
\def\cbig#1.{\overline{\operatorname {B}(#1)}}
\def\coef#1.{\frac{(#1-1)}{#1}}
\def\vit#1.{D_{\langle #1 \rangle}}
\def\mm#1.{\overline {M}_{0,#1}}
\def\Hone#1.{H^1(#1,{\ring #1.})}
\def\ac#1.{\overline {\mathbb F}_{#1}}

\def\adj#1.{\frac {#1-1}{#1}}
\def\spn#1.{\overline{#1}}
\def\pek#1.#2.{\Cal P^{#1}(#2)}
\def\plk#1.#2.{\Cal P^{\leq #1}(#2)}
\def\ev#1.{\operatorname{ev_{#1}}}
\def\ilist#1.{{#1}_1,{#1}_2,\dotsc}
\def\bminv#1.{(\nu_1,s_1;\nu_2,s_2;\dots ;\nu_{#1},s_{#1};\nu_{r+1})}
\def\zinv#1.{(\nu_1,s_1;\nu_2,s_2;\dots ;\nu_{#1},s_{#1};0)}
\def\iinv#1.{(\nu_1,s_1;\nu_2,s_2;\dots ;\nu_{#1},s_{#1};\infty)}
\def\scr#1.{\mathbf{\EuScript{#1}}}
\def\mg#1.{\overline {M}_{#1}}
\def\inter#1.{\underset #1{\cdot}}
\def\cate#1.{\text{(\underline{#1})}}
\def\dls#1.{\overrightarrow{#1}}
\def\id#1.{\text{id}_{#1}}

\def\Diff{\operatorname{Diff}}

\def\Hom{\operatorname{Hom}}

\def\Proj{\operatorname{Proj}}

\def\Supp{\operatorname{Supp}}

\def\coeff{\operatorname{coeff}}
\def\dim{\operatorname{dim}}

\def\deg{\operatorname{deg}}

\def\im{\operatorname{Im}}

\def\Exc{\operatorname{Exc}}

\def\Div{\operatorname{Div}}

\def\mult{\operatorname{mult}}

\def\rest{\operatorname{res}}

\def\vol{\operatorname{vol}}

\def\C`har{\operatorname{char}}

\def\lct{\operatorname{lct}}

\def\C{\mathbb C}

\def\e{\Cal E}

\def\e1{E_1}
\def\e2{E_2}

\def\mapdown#1{\big\downarrow\rlap{$\vcenter
{\hbox{$\scriptstyle#1$}}$}}

\def\mapse#1{
{\vcenter{\hbox{$\mathop{\smash{\raise1pt\hbox{$\diagdown$}\!\lower7pt
\hbox{$\searrow$}}\vphantom{p}}\limits_{#1}\vphantom{\mapdown{}}$}}}}

\def\VR#1.{height#1pt&\omit&&\omit&&\omit&&\omit&&\omit&\cr}

\def\VRT#1.{height#1pt&\omit&&\omit&\cr}

\makeatletter
\renewcommand*\env@matrix[1][*\c@MaxMatrixCols c]{%
  \hskip -\arraycolsep
  \let\@ifnextchar\new@ifnextchar
  \array{#1}}
\makeatother

\begin{document}
\title[Boundedness of varieties]{Boundedness of varieties of log general type}
\date{\today}
\author{Christopher D. Hacon}
\address{Department of Mathematics \\
University of Utah\\
155 South 1400 East\\
JWB 233\\
Salt Lake City, UT 84112, USA}
\email{hacon@math.utah.edu}
\author{James M\textsuperscript{c}Kernan}
\address{Department of Mathematics\\
University of California, San Diego\\
9500 Gilman Drive \# 0112\\
La Jolla, CA  92093-0112, USA}
\email{jmckernan@math.ucsd.edu}
\author{Chenyang Xu}
\address{BICMR\\
Peking University\\
Beijing, 100871, China}
\email{cyxu@math.pku.edu.cn}

\thanks{The first author was partially supported by DMS-1300750, DMS-1265285 and a grant
  from the Simons foundation, the second author was partially supported by NSF research
  grant no: 0701101, no: 1200656 and no: 1265263 and this research was partially funded by
  the Simons foundation and by the Mathematische Forschungsinstitut Oberwolfach and the
  third author was partially supported by The National Science Fund for Distinguished
  Young Scholars (11425101) grant from China.  Part of this work was completed whilst the
  second and third authors were visiting the Freiburg Institute of Advanced Studies and
  they would like to thank Stefan Kebekus and the Institute for providing such a congenial
  place to work.  We would also like to thank the referee for a careful reading of the
  paper and for some very useful comments.}

\begin{abstract} We survey recent results on the boundedness of the moduli functor of stable pairs.
\end{abstract}

\maketitle

\tableofcontents

\makeatletter
\renewcommand{\thetheorem}{\thesubsection.\arabic{theorem}}
\@addtoreset{theorem}{subsection}
\makeatother

\section{Introduction} 

The purpose of this paper is to give an overview of recent results on the moduli of
varieties of general type.  We start with a gentle introduction to the subject, reviewing
the case of curves and surfaces as motivation for some of the definitions.  Then we switch
gears a little and collect together in one place an account of boundedness of the moduli
functor.  None of the results here are new but we thought it would be useful to present
them together, as currently they are spread over several papers.  We also take this
opportunity to present an alternative argument for one step of the proof.  Due to
constraints imposed by space, we do not give full details for many of the proofs; anyone
wanting to see more details should look at the original papers.

The theory of moduli in higher dimensions is extremely rich and interesting and so we are
obliged to skip many interesting topics, which are fortunately covered in the many
excellent surveys and books, see for example \cite{Kollar10a}, \cite{Kollar13} and
\cite{Kovacs09}.  We focus on two aspects of the construction, what we need to add to get
a compact moduli space and how to prove boundedness.  We start with what we should add to
get a compact moduli space.

The moduli space $M_g$ of smooth curves of genus $g\geq 2$ is a quasi-projective variety
of dimension $3g-3$.  The moduli space of stable curves $\mg g.$ is a geometrically
meaningful compactification of $M_g$, so that $\mg g.$ is projective and $M_g$ is an open
subset.  Geometrically meaningful refers to the fact the added points correspond to
geometric objects which are as close as possible to the original objects.  In the case of
$M_g$ we add stable curves $C$, connected curves of arithmetic genus $g$, with nodal
singularities, such that the automorphism group is finite, or better (and equivalently),
the canonical divisor $K_C$ is ample.

We adopt a similar definition of stable in higher dimensions.
\begin{definition} A {\bf semi log canonical model} $(X,B)$ is a projective semi log
canonical pair (cf. \S \ref{ss-nc}) such that $K_X+B$ is ample.  Fix $n\in \mathbb{N}$,
$I\subset [0,1]$ and $d\in \mathbb{R} _{>0}$.  Let $\mathfrak{F}_{\rm slc}(n,I,d)$ be the
set of all $n$-dimensional semi log canonical models such that the coefficients of $B$
belong to $I$ (that is, $\coeff(B)\subset I$), $K_X+B$ is $\mathbb{Q}$-Cartier and
$(K_X+B)^n=d$.
\end{definition}

We now attempt to give some motivation for the admittedly technical definition of semi log
canonical models.

\subsection{Semi log canonical models} 

There are in general many degenerations of the same family of varieties.  Given a moduli
problem properness corresponds to existence and uniqueness of the limit.  Given a family
of smooth curves there is a unique stable limit, as proved by Deligne and Mumford
\cite{DM69}.

We review the construction of the stable limit.  Let $f\colon\map X^0.C^0.$ be a family of
smooth curves of genus $g\geq 2$ over a smooth curve $C^0=C\setminus 0$ where $C$ is an
affine curve and $0$ is a closed point.  By semistable reduction, after replacing $C^0$ by
an appropriate base change, we may assume that there is a proper surjective morphism
$f\colon \map X.C.$ such that $X$ is smooth and the central fibre $X_0$ is reduced with
simple normal crossings.  The choice of $X_0$ is not unique, since we are free to blow up
the central fibre.  So we run the minimal model program over $C$, contracting $-1$-curves
(that is, curves $E\cong \pr 1.$ such that $E^2=K_X\cdot E=-1$) in the central fibre.  We
end with a relative minimal model $\map X^m.C.$, so that $X^m$ is smooth and $K_{X^m/C}$
is nef over $C$.  If we further contract all $-2$-curves, that is, curves $E\cong \pr 1.$
such that $E^2=-2$ and $K_X\cdot E=0$ then we obtain the relative canonical model
$\map X^c.C.$.  The model $X^c$ is characterised by the fact that it has Gorenstein
canonical (aka Du Val, aka ADE) singularities and $K_{X^c/C}$ is ample over $C$.

A key observation is that we can construct the relative canonical model directly as 
\[
X^c=\Proj_CR(X,K_X) \qquad \text{where} \qquad R(X,K_X)=\bigoplus_{m\geq 0}H^0(X,\ring X.(mK_X))
\]
is the canonical ring; note that since $C$ is affine, $H^0(X,\ring X.(mK_X))$ can be
identified with the $\ring C.$-module $f_*\ring X.(mK_X)$.  Observe that $X^c$ is
isomorphic to $X$ over $C^0$.  Since the relative canonical model is unique, it follows
that the family above has a unique compactification to a family of stable curves (that is,
the moduli functor of stable curves is proper).

Here is another instructive example (cf. \cite{Kollar10a}).  Let $S$ be any smooth
projective surface such that $K_S$ is ample.  Consider the family $X=S\times \af 1.$ and
three sections $C_i$ with $i\in \{1,2,3\}$ which meet as transversely as possible in a
point $(p,0)\in S\times 0$.  Blowing up the $C_i$ in different orders we obtain two
families $X^1$ and $X^2$ which are isomorphic over $\af 1. \setminus 0$ but have distinct
central fibres $X^1_0\ne X^2_0$.  Therefore the corresponding moduli functor is not
proper.  If however we only consider canonical models, then this problem does not appear
since the relative canonical model
\[
\Proj R(X^i,K_{X^i})=\Proj (\bigoplus _{m\geq 0} H^0(X^i,\ring X^i.(mK_{X^i})))\cong X
\]
is unique.

Properness of the moduli functor of semi log canonical models is established in
\cite{HX11}.  The proof is similar to the argument sketched above for stable curves,
except that an ad hoc argument is necessary to construct the relative canonical model, as
the minimal model program for semi log canonical pairs is only known to hold in special
cases.

The moduli space $M_{g,n}$ of smooth curves $C$ of genus $g$ with $n$ points $\llist p.n.$
is a natural generalisation of the moduli space of curves.  It has a natural
compactification $\mgn g.n.$, the moduli space of stable curves of genus $g$ with $n$
points.  The points of $\mgn g.n.$ correspond to connected nodal curves with $n$ labelled
points $\llist p.n.$ which are not nodes such that $K_S+\Delta$ is ample, where $\Delta$
is the sum of the labelled points.  Therefore a stable pointed curve is the same as a semi
log canonical model (up to ordering the labelled points), with coefficient set $I=\{1\}$.

There are many reasons to consider labelled points.  $\mgn 0.n.$ is a non-trivial moduli
space with a very interesting geometry and yet it is given by an explicit blow up of
projective space.  On the other hand, allowing the coefficients of $\Delta$ to vary, so
that we take different choices for the coefficient set $I$, gives a way to understand the
extremely rich geometry of the moduli space of curves.  For different choices of $I$ we
get slightly different moduli problems and so we get different birational models of
$\mgn g.n.$, \cite{Hassett02}.  Finally the normalisation of a stable curve is a stable
pointed curve such that the inverse image of the nodes are labelled points.  Studying
stable pointed curves offers an inductive way to understand the geometry of $\mg g.$.

There is a similar picture in higher dimensions.  We know of the existence of a moduli
space of semi log canonical models in many cases.  We will sometimes refer to this space
as the KSBA compactification (constructed by Koll\'ar, Shepherd-Barron and Alexeev).  If
$S$ is a cubic surface in $\pr 3.$ then $K_S+\Delta=-8K_S$ is ample and log canonical,
where $\Delta$ is the sum of the twenty seven lines, so that $(X,\Delta)$ is a semi log
canonical model.  Therefore a component of the KSBA compactification with $I=\{1\}$ gives
a moduli space of cubic surfaces, \cite{HKT09}.  If $C$ is a smooth plane curve of degree
$d>3$ then $K_S+tC$ is ample for any $t>3/d$.  Therefore a component of the KSBA
compactification for suitable choice of coefficient set $I$ gives a compactification of
the moduli space of plane curves of degree $d$, \cite{Hacking04}.  On the other hand, if
we allow the coefficients of $\Delta$ to vary then this induces birational maps between
moduli spaces and we can connect two moduli spaces by a sequence of such transformations.

If $(X,\Delta)$ is a semi log canonical pair then $X$ is in general not normal.  If
$\nu\colon\map Y.X.$ is the normalisation then we may write 
\[
K_Y+\Gamma=\nu^*(K_X+\Delta).
\]
The divisor $\Gamma$ is the strict transform of $\Delta$ plus the double locus taken with
coefficient one.  If $\Delta=0$ then $\Gamma=0$ if and only if $X$ is normal.  The pair
$(Y,\Gamma)$ is log canonical and it is a disjoint union of log canonical pairs
$(Y_i,\Gamma_i)$.  The pair $(X,\Delta)$ is obtained from $(Y,\Gamma)$ by an appropriate
idenfication of the double locus.  By a result of Koll\'ar (cf. \cite{Kollar11} and
\cite[5.13]{Kollar10a}), if $(X,\Delta)$ is a semi log canonical model, then it can be
recovered from the data of $(Y,\Gamma)$ and an involution of the double locus (that is,
the components of $\Gamma$ which do not correspond to components of $\Delta$).

We have already seen that it is interesting to allow the coefficients of $\Delta$ to be
fractional.  It is also useful when trying to establish boundedness by induction on the
dimension.  For example if $(X,\Delta=S+B)$ is a log canonical pair, the coefficients of
$\Delta$ are all one and $S$ is a prime divisor which is a component of $\Delta$ then by
adjunction (cf. Theorem \ref{tSLA}) 
\[
(K_X+B)|_S=K_S+\Diff_S(B-S)
\]
where $(S,\Diff_S(B-S))$ is a log canonical pair and the coefficients of $\Diff_S(B-S)$
belong to 
\[
J=\{\, 1-\frac 1n \,|\, n\in \mathbb{N} \,\} \cup \{1\}.
\]
In fact the coefficients of $\Diff_S(B-S)$ belong to $J$ whenever the coefficients of $B$
belong $J$.  As $J$ is the smallest set containing $1$ closed under taking the different,
the set of coefficients $J$ is sometimes called the \textbf{standard coefficient set}.

Note that the set $J$ is not finite, however it satisfies the {\bf descending chain
  condition} (or DCC condition), that is, every non increasing sequence is eventually
constant.  To prove boundedness it is convenient to work with any coefficient set
$I\subset [0,1]$ which satisfies the DCC.

We note that there is one aspect of the theory of moduli in higher dimensions which is
quite different from the case of curves.  The moduli space $\mg g.$ of curves is
irreducible.  Moreover $M_g$ is a dense open subset.  However even if we take
$I=\emptyset$ and fix $d$ the KSBA moduli space might have more than one component and no
point of these components corresponds to a normal surface.

\subsection{Main Theorems}

Our main result (\cite{Alexeev94, AM04} for the surface case and \cite{HMX14} in general)
is the following.
\begin{theorem}\label{t-bound} Fix $n\in \mathbb{N}$, a set $I\subset [0,1]\cap \mathbb{Q}$ 
satisfying the DCC and $d>0$.  Then the set $\mathfrak{F}_{\rm slc}(n,I,d)$ is bounded,
that is, there exists a projective morphism of quasi-projective varieties
$\pi\colon\map \mathcal{X}.T.$ and a $\mathbb{Q}$-divsor $\mathcal{B}$ on $\mathcal{X}$
such that the set of pairs $\{\,(\mathcal{X} _t,\mathcal{B}_t)\,|\,t\in T\,\}$ given by
the fibres of $\pi$ is in bijection with the elements of $\mathfrak{F}_{\rm slc}(n,I,d)$.
\end{theorem} 
The above result is equivalent to:
\begin{theorem}\label{tcbound} Fix $n\in \mathbb{N}$, a set $I\subset [0,1]\cap \mathbb{Q}$ 
satisfying the DCC and $d>0$.  Then there is an integer $r=r(n,I,d)$ such that if
$(X,B)\in \mathfrak{F}_{\rm slc}(n,I,d)$ then $r(K_X+B)$ is Cartier and very ample.

In particular, the coefficients of $B$ always belong to a finite set $I_0\subset I$.
\end{theorem}

One of the main results necessary to prove the previous theorem is the following, which
was conjectured in \cite{Kollar92b, Alexeev94}.
\begin{theorem}\label{t-vol} Fix $n\in \mathbb{N}$ and a set $I\subset [0,1]\cap \mathbb{Q}$
satisfying the DCC.  Let
\[
\mathcal{V}(n,I)=\{\, d=(K_X+B)^n\,|\, (X,B)\in \mathfrak{F}_{\rm slc}(n,I)\,\}
\]
be the set of all possible volumes of semi log canonical models of dimension $n$ with
coefficients belonging to $I$.

Then $\mathcal{V}(n,I)$ satisfies the DCC.  In particular it has a minimal element
$v(n,I)>0$.
\end{theorem}

If $\dim X=1$ then $X$ is a curve and
\[
\vol(X,K_X+B)=\deg (K_X+B)=2g-2 +\sum b_i,
\] 
where $g$ is the arithmetic genus of $X$ and $B=\sum b_iB_i$.  Thus, the set 
\[
\mathcal V (1, I)=\{2g-2+\sum b_i |b_i\in I \}\cap \mathbb{R} _{>0}
\]
of possible volumes satisfies the DCC.  For example, if $I$ is empty, then
$v(1,\emptyset )=2$ and if $J$ is the set of standard coefficients, then it is well known
that $v(1,J)=1/42$.  Finally, if $I=\{ 0,1 \}$, one sees that $mK_X$ is very ample for all
$m\geq 3$, as an easy consequence of Riemann Roch.

If $\dim X=2$, and $X$ has canonical singularities then the canonical divisor is Cartier
and in particular $K_X^2\in \mathbb{N}$ so that $K_X^2 \geq 1$.  By a result of Bombieri,
it is also known that $mK_X$ is very ample for $m\geq 5$ \cite{Bombieri70},
\cite{Bombieri73} (a similar result also follows in positive characteristic
\cite{Ekedahl88}).  On the other hand $\mathcal V(2,I)$ is hard to compute and there are
no explicit bounds known for $r(2,\emptyset , d)$.

If $\dim X =3$ then there are semi log canonical models with canonical singularities of
arbitrarily high index, therefore there is no integer $r>0$ such that $rK_X$ is very ample
for any 3-dimensional canonical model.  Since $K_X$ is not necessarily Cartier, the volume
$K_X^3$ may not be an integer and in particular it may be smaller than 1.  In fact by
\cite{Fletcher00} a general hypersurface $X$ of degree 46 in weighted projective space
$\proj 4,5,6,7,23.$ has volume $K_X^3=1/420 $ and $mK_X $ is birational for $m=23$ or
$m\geq 27$.  On the other hand, using Reid's Riemann-Roch formula, it is shown in
\cite{CC10}, \cite{CC12} that $K_X^3\geq 1/1680$ and $rK_X$ is birational for $r\geq 61$
for any $3$-dimensional canonical model.

\subsection{Boundedness of canonical models}

In general the problem of determining lower bounds for the volume of $K_X$ and which
multiples $mK_X$ of $K_X$ that are very ample is not easy. The first general result for
canonical models in arbitrary dimension is based on ideas of Tsuji (\cite{Tsuji07},
\cite{HM05b} and \cite{Takayama06}).
\begin{theorem}\label{t-Tsuji} Fix $n\in\mathbb{N}$ and $d>0$.  Then
\begin{enumerate} 
\item The set of canonical volumes $\mathcal W (n)=\{K_X^n \}$ where $X$ is a
$n$-dimensional canonical model, is discrete.  In particular the minimum $w=w(n)$ is
achieved.
\item There exists an integer $k=k(n)>0$ such that if $X$ is an $n$-dimensional canonical
model, then $mK_X$ is birational for any $m\geq k$.
\item There exists an integer $r=r(n,d)>0$ such that if $X$ is an $n$-dimensional
canonical model with $K_X^n=d$, then $rK_X$ is very ample.
\end{enumerate}
\end{theorem}

Note that it is not the case that the volumes of $d$-dimensional log canonical models is
discrete, in fact by examples of \cite[36]{Kollar08}, they have accumulation points from
below.

\begin{proof}[Sketch of the proof of Theorem \ref{t-Tsuji}] Tsuji's idea is to first prove
the following weaker version of (2):
\begin{claim}\label{c-Tsuji} There exist constants $A$, $B>0$ such that $mK_X$ is birational
for any $m\geq A(K_X^n)^{-1/n}+B$.
\end{claim}
To prove the claim, it suffices to show that for very general points $x,y\in X$ there is
an effective $\mathbb{Q}$-divisor $D$ such that
\begin{enumerate}
\item $D\sim _\mathbb{Q} \lambda K_X$ where $\lambda < A(K_X^n)^{-1/n}+B-1$,
\item $\mathcal{J}(X,D)_x=\mathfrak{m} _x$ in a neighbourhood of $x\in X$ and
\item $\mathcal{J}(X,D)\subset\mathfrak{m}_y$.
\end{enumerate}

Therefore $x$ is an isolated point of the cosupport of $\mathcal J (X,D)$ and $y$ is
contained in the cosupport of $\mathcal J (X,D)$.  Applying Nadel vanishing we obtain
$H^1(X,\omega _X^m\otimes \mathcal{J} (X,D))=0$ for any integer
$m\geq A(K_X^n)^{-1/n}+B-1$ and so there is a surjection
\[
\map H^0(X,\omega _X^m).{H^0(X,\omega _X^m\otimes \ring X./\mathcal{J}(X,D))}..
\]
By our assumptions, $\mathcal O_X/\mathcal{J} (X,D)=\mathcal F \oplus \mathcal G$ where
$\Supp(\mathcal F)=x$ and $y\in \Supp(\mathcal G)$.  From the surjection
$\map \mathcal F.{\ring X./\frak m _x}.\cong \mathbb{C}(x)$ it easily follows that there
exists a section of $\omega _X^m$ vanishing at $y$ and not vanishing at $x$. Therefore
$|mK_X|$ induces a birational map.

We now explain how to produce the divisor $D$.  We focus on establishing the condition
$\mathcal J (X,D)_x=\frak m _x$ and we ignore the condition
$\mathcal J (X,D)\subset \frak m_y$ since this is easier.  Fix $0<\epsilon \ll 1$.  Since
\[
h^0(X,\ring X.(tK _X))=\frac{K_X^n}{n!}t^n+o(t^n)
\]
and vanishing at $x$ to order $s$ imposes at most $s^n/n!+o(s^{n})$ conditions, for every
$l\gg 0$ there is a section $D_l\in |lK_X|$ with
\[
\mult_x(D_l)> l((K_X^n)^{1/n}-\epsilon).
\]  
If $D=\lambda D_l/l\sim _\mathbb{Q} \lambda K_X$ where $\lambda =\lct_x(X;D_l/l)$, then
\[
\lambda \leq n/l((K_X^n)^{1/n}-\epsilon)
\] 
so that $\lambda \leq A' (K_X^n)^{-1/n}+B'$ for appropriate constants $A',B'$ depending
only on $n$.  By definition of $D$, we have $\mathcal{J} (X,D)\subset \frak m_x$. Let
$x\in V\subset X$ be an irreducible component of the co-support of $\mathcal{J} (X,D)$. By
standard arguments (see Proposition \ref{p-mincentre}), we may assume that $V$ is the only
such component. If $\dim V =0$, then $V=x$ and we are done, so suppose that $n'=\dim V>0$.

Since $x\in X$ is very general then $V$ is of general type.  Let $\nu\colon\map V'.V.$ be
a log resolution. Then by induction on the dimension, there exists a constant $k'=k(n-1)$
such that $\phi _{k'K_{V'}}\colon\rmap V'.{\pr M.}.$ is birational. Let $n'=\dim V$. Pick
$x'\in V'$ a general point and 
\[
D_{1,V'}=\frac {n'}{n'+1}(H_1+\ldots +H_{n'+1})
\]
where $H_i\in |k'K_{V'}|$ are divisors corresponding to general hyperplanes on
$\mathbb P ^M$ containing $\phi _{k'K_{V'}}(x')$.  Let $D_{1,V}=\nu _*D_{1,V'}$.  It is
easy to see that $x'$ is an isolated non Kawamata log terminal centre of $(V,D_{1,V})$
(with a unique non Kawamata log terminal place).

Assume for simplicity that $V$ is normal.  By Kawamata subadjunction, it follows that
\[
(1+\lambda )K_X|_V-K_V \sim_{\mathbb{R}} (K_X+D)|_V-K_V
\]
is pseudo-effective.  Since $K_X$ is ample, for any $\delta >0$, we may assume that there
is an effective $\mathbb{R}$-divisor
\[
D^*_{1,V}\sim _\mathbb{R} (1+\lambda)( \frac {n'k}{n'+1}+\delta)K_X|_V
\]
such that $x'$ is an isolated non Kawamata log terminal centre of $(V,D^*_{1,V})$ (with a
unique non Kawamata log terminal place).  By Serre vanishing there is a divisor
\[
D_1\sim _\mathbb{R} (1+\lambda)( \frac {n'k}{n'+1}+\delta) K_X
\] 
such that $D_1|_V=D_{1,V}$.  By inversion of adjunction $x'$ is a minimal non Kawamata log
terminal centre of $ (X,D+D_1)$.  After perturbing $D'=D+D_1$ we may assume that
$\mathcal{J}(X,D')=\frak m_{x'}$ in a neighbourhood of $x'\in V \subset X$. Note that there
exist constants $A'',B''>0$ such that
\[
\lambda + (1+\lambda)( \frac {n'k}{n'+1}+\delta)\leq A''(K_X^n)^{-1/n}+B''.
\]

Finally we sketch Tsuji's argument showing that Corollary \ref{c-Tsuji} implies Theorem
\ref{t-Tsuji}.  Let $m_0=\lceil A(K_X^n)^{-1/n}+B \rceil$ and $Z$ be the image of $X$ via
$|m_0K_X|$. Then $Z$ is birational to $X$.  Fix any $M>0$. If $K_X^n\geq M$ then (2) of
Theorem \ref{t-Tsuji} holds with $k=\lceil A(M)^{-1/n}+B \rceil$.  Therefore suppose that
$K_X^n< M$. In this case we have
\begin{align*} 
\deg (Z)&\leq m_0^nK_X^n\\
        &<(A(K_X^n)^{-1/n}+B+1)^nK_X^n\\
        &\leq (A+(B+1)M)^n.
\end{align*} 

Therefore $X$ is birationally bounded.  More precisely, using the corresponding Chow
variety, we obtain a projective morphism of quasi-projective varieties
$\map\mathcal{Z}.T.$ such that for any $X$ as above there exists a point $t\in T$ and a
birational map $\rmap X.\mathcal{Z}_t.$.  Let $\map \mathcal{Z}'.\mathcal{Z}.$ be a
resolution. After decomposing $T$ (and $\mathcal{Z}$) into a disjoint union of locally
closed subsets, we may assume that $\map \mathcal{Z}'.T.$ is a smooth morphism.  We may
also assume that the subset of points $t\in T$ such that $\mathcal{Z}'_t$ is a variety of
general type, is dense in $T$.  By Siu's theorem on the deformation invariance of
plurigenera, we may then assume that all fibres $\mathcal{Z}'_t$ are varieties of general
type and that there are finitely many possible volumes
\[
K_X^n=\vol (\mathcal{Z}_t,K_{\mathcal{Z}_t})=\lim \frac{h^0(\mathcal{Z}_t,mK_{\mathcal{Z}_t})}{m^n/n!}.
\]
This implies (1) of Theorem \ref{t-Tsuji}.  It is also clear that (2) of Theorem
\ref{t-Tsuji} holds with $k=\lceil A(w(n))^{-1/n}+B \rceil$.  To prove (3), assume that
$d< M$.  After throwing away finitely many components of $T$, we may assume that
$\vol (\mathcal{Z}_t,K_{\mathcal{Z} _t})=d$ for all $t\in T$.  Let $\map \mathcal{X}.T.$
be the relative canonical model of $\mathcal{Z}/T$ which exists by \cite{BCHM10}.  Since
$K_{\mathcal{X}}$ is relatively ample, it follows that there is an integer $r$ such that
$rK_{\mathcal{X}}$ is relatively very ample and hence $rK_{\mathcal{Z} _t}$ is very ample
for all $t\in T$. Therefore (3) Theorem \ref{t-Tsuji} also holds. \end{proof} 

It is natural to try and generalize the above argument to the case of log pairs. Not
surprisingly there are many technical difficulties.  The first obvious difficulty is that
it is no longer sufficient to prove the birational boundedness of varieties but we need to
prove some version of birational boundedness for log pairs, cf. \eqref{d-bounded}.  

The basic structure of the proofs of Theorems \ref{t-bound}, \ref{tcbound} and \ref{t-vol}
is similar to that of Theorem \ref{t-Tsuji}. The proof can be divided into three steps
(see \cite{HMX10}, \cite{HMX12}, \cite{HMX14}).

In the first step, we want to show that if we have a class $\mathfrak{D}$ of
$n$-dimensional log canonical pairs which is birationally bounded and with all the
coefficients belonging to a fixed DCC set $I$, then the set
\[
\{\, \vol(X,K_X+B) \,|\, (X,B)\in \mathfrak{D} \,\}
\]
also satisfies the DCC.  Under suitable smoothness assumptions, we obtain a version of
invariance of plurigenera for pairs.  Using this, we can easily reduce to the case that
$T$ is a point in the definition of a log birationally bounded family, that is, we can
assume all pairs are birational to a fixed pair.  Then there is a lengthy combinatorial
argument, mainly using toroidal geometry calculations, to finish the argument.

In the second step, we want to prove that all log general type pairs in $\mathfrak{D}$
with volume bounded from above form a log birationally bounded family.  This step is
similar to the proof of Theorem \ref{t-Tsuji} (unluckily many difficulties arise due to
the presence of the boundary).  This is done in \cite{HMX14} via a complicated induction
which relies on the ACC for log canonical thresholds and other results.  We adopt a more
direct approach here, where we first prove the result for coefficient sets $I$ of the form
\[
\{\, \frac ip \,|\, 0\leq i\leq p \,\}
\]
and then deduce the general case.

In the final step we deduce boundedness from log birational boundedness.  This is a direct
consequence of the Abundance Conjecture.  In our situation, we are able to use a
deformation invariance of plurigenera for pairs (proved with analytic methods by
Berndtsson and P\u aun) to establish the required special case of the abundance
conjecture.

section{Preliminaries}

\subsection{Notation and conventions}\label{ss-nc}

We work over the field of complex numbers $\mathbb{C}$.  A {\bf pair} $(X,B)$ is given by
a normal variety $X$ and an effective $\mathbb{R}$-divisor $B=\sum _{i=1}^kb_iB_i$ such
that $K_X+B$ is $\mathbb{R}$-Cartier.  We denote the coefficients of $B$ by
$\coeff (B)=\{b_1,\ldots , b_k\}$.  Let $q\in \mathbb{N}$,
$I_0=\{\, \frac jq \,|\, 1\leq j\leq q \,\}$.  We say that $D(I_0)$ is a {\bf
  hyperstandard set of coefficients}.  Observe that for any finite set of rational numbers
$J_0\subset [0,1]$, we can find $q\in \mathbb{N}$ such that $J_0\subset I_0$.

We let $\lfloor B\rfloor=\sum \lfloor b_i\rfloor B_i$ where $\lfloor b\rfloor$ is the
greatest integer $\leq b$ and $\{ B\}=B-\lfloor B\rfloor$.  The {\bf support} of $B$ is
given by $\Supp(B)=\cup _{b_i\neq 0}B_i$.  The {\bf strata} of $(X, B)$ are the
irreducible components of intersections
\[
B_I=\cap _{j\in I} B_j=B_{i_1}\cap \ldots \cap B_{i_r},
\]
where $I=\{\, \llist i.r.\,\}$ is a subset of the non-zero coefficients, including the
empty intersection $X=B_\emptyset$.  If $B'=\sum b_i'B_i$ is another $\mathbb{R}$-divisor,
then $B\wedge B'=\sum (b_i\wedge b'_i)B_i$ and $B\vee B'=\sum (b_i\vee b'_i)B_i$ where
$b_i\wedge b'_i=\min \{ b_i,b'_i\}$ and $b_i\vee b'_i=\max \{ b_i,b'_i\}$.

We say the pair $(X,B)$ is a {\bf toroidal pair} if the inclusion $\emb U.X.$ of the
complement $U$ of the support of $B$ is toroidal, so that locally, in the analytic
topology, the inclusion is isomorphic to the inclusion of the open torus inside a toric
variety, see \cite{KKMS73}.

For any proper birational morphism $\nu\colon\map X'.X.$, we pick a canonical divisor
$K_{X'}$ such that $\nu _* K_{X'}=K_X$ and we write 
\[
K_{X'}+B'=\nu ^*(K_X+B)+\sum a_{E_i}E_i
\]
where $B'$ is the strict transform of $B$.  The numbers $a_{E_i}=a_{E_i}(X,B)$ are the
{\bf discrepancies} of $E_i$ with respect to $(X,B)$, the \textbf{discrepancy} of $(X,B)$
is $\inf \{a_{E}(X,B)\}$ where $E$ runs over all divisors over $X$ and the {\bf total
  discrepancy} $a(X,B)$ of $(X,B)$ is the minimum of the discrepancy and $\coeff)(-B)$.
We say that $(X,B)$ is {\bf Kawamata log terminal} (resp. {\bf log canonical}, {\bf
  terminal}) if $a(X,B)>-1$ (resp. $a(X,B)\geq -1$, $a_E(X,B)>0$ for any divisor $E$
exceptional over $X$).  Note that to check if a pair is either Kawamata log terminal or
log canonical it suffices to check what happens on a single log resolution, that is, on a
proper birational morphism $\nu\colon\map X'.X.$ such that the exceptional locus is a
divisor and $\nu ^{-1}_* B+\Exc(\nu )$ has simple normal crossings.  A divisor $E$ over
$X$ is a {\bf non Kawamata log terminal place} of $(X,B)$ if $a_E(X,B)\leq -1$. The image
of a non Kawamata log terminal place $E$ in $X$ is a {\bf non-Kawamata log terminal
  centre}.  {\bf Non log canonical places} and {\bf centres} are defined similarly by
requiring $a_E(X,B)<-1$.  A pair $(X,B)$ is {\bf divisorially log terminal} if it is log
canonical and there is an open subset $U\subset X$ containing the generic points of all
non Kawamata log terminal centres such that $(U,B|_U)$ has simple normal crossings.  In
this case, by a result of Szab{\'o}, it is known that there exists a log resolution of
$(X,B)$ which is an isomorphism over $U$.

If $(X,B)$ is a log canonical pair and $D\geq 0$ is an effective $\mathbb{R}$ divisor,
then we define the {\bf log canonical threshold} of $(X,B)$ with respect to $D$ by
\[
\lct (X,B;D)=\sup \{\, t\geq 0 \,|\, \text{$(X,B+tD)$ is log canonical} \,\}
\]
For any closed point $x\in X$, $\lct _x(X,B;D)$ will denote the log canonical threshold
computed on a sufficiently small open subset of $x\in X$.  In particular,
$(X,B+\lambda D)$ is log canonical in a neighbourhood of $x\in X$ and the non-Kawamata log
terminal locus of $(X,B+\lambda D)$ contains $x$ where $\lambda=\lct _x(X,B;D)$.

Let $X$ be a normal variety and consider the set of all proper birational morphisms
$f\colon\map Y.X.$ where $Y$ is normal.  We have natural maps
$f_*\colon\map \Div(Y).\Div(X).$.  The space {\bf b-divisors} is
\[
{\bf Div}(X)=\lim _{\{Y\to X\}}\Div(Y).
\]
Note that $f_*$ induces an isomorphism ${\bf Div}(Y)\cong {\bf Div}(X)$ and that an
element $\mathbf B\in {\bf Div}(X)$ is specified by the corresponding traces
$\mathbf B _Y$ of $\mathbf B$ on each birational model $\map Y.X.$.  If $E$ is a divisor
on $Y$, then we let $\mathbf B (E)=\mult_E(\mathbf B _Y)$.  Given a log pair $(X,B)$ and a
proper birational morphism $f\colon\map X'.X.$, we may write $K_{X'}+B_{X'}=f^*(K_X+B)$.
We define the b-divisors $\mathbf{L} _B$ and $\mathbf{M} _B$ as follows
\[
\mathbf{M} _{B,X'}=f ^{-1}_*B+\Exc(f)\qquad \text{and}\qquad \mathbf{L}_{B,X'}=B_{X'}\vee 0.
\]

A {\bf semi log canonical pair} (SLC pair) $(X,B)$ is given by an $S_2$ variety whose
singularities in codimension $1$ are nodes and an effective $\mathbb{R}$-divisor $B$ none
of whose components are contained in the singular locus of $X$ such that if
$\nu \colon\map X^\nu.X.$ is the normalisation and $K_{X^\nu}+B^\nu=\pi ^*(K_X+B)$, then
each component of $(X^\nu,B^{\nu})$ is log canonical.  A {\bf semi log canonical model}
(SLC model) is a projective SLC pair $(X,B)$ such that $K_X+B$ is ample.

If $X$ is a smooth variety and $D$ is an effective $\mathbb{R}$-divisor on $X$, then {\bf
  the multiplier ideal sheaf} is defined by
\[
\mathcal{J}(X,D)=\mu _*(K_{X'/X}-\lfloor \mu ^* D \rfloor )\subset \ring X.
\]
where $\mu \colon\map X'.X.$ is a log resolution of $(X,D)$.  It is known that the
definition does not depend on the choice of a log resolution and
$\mathcal{J}(X,D)=\ring X.$ if and only if $(X,D)$ is Kawamata log terminal and in fact
the support of $\ring X./\mathcal J (X,D)$ (that is, the {\bf co-support} of
$\mathcal J(X,D)$) is the union of all non Kawamata log terminal centres of $(X,D)$.  Note
that
\[
\lct (X,B;D)=\sup \{\, t\geq 0 \,|\, \mathcal{J}(X,B+tD)=\ring X. \,\}.
\]
We refer the reader to \cite{Lazarsfeld04b} for a comprehensive treatment of multiplier
ideal sheaves and their properties.

Let $\pi \colon\map X.U.$ be a morphism, then $\pi$ is a {\bf contraction morphism} if and
only if $\pi _*\ring X.=\ring U.$.  If $f\colon\map X.U.$ is a morphism and $(X,B)$ is a
pair, then we say that $(X,B)$ is {\bf log smooth over} $U$ if $(X,B)$ has simple normal
crossings and every stratum of $(X,\Supp(B))$ (including $X$) is smooth over $U$.

A {\bf birational contraction} $f\colon\rmap X.Y.$ is a proper birational map of normal
varieties such that $f^{-1}$ has no exceptional divisors.  If $p\colon\map W.X.$, and
$q\colon\map W.Y.$ is a common resolution then $f$ is a birational contraction if and only
if every $p$-exceptional divisor is $q$-exceptional.  If $D$ is an $\mathbb{R}$-Cartier
divisor on $X$ such that $f_*D$ is $\mathbb{R}$-Cartier on $Y$ then $f$ is $D$-{\bf
  non-positive} (resp. $D$-{\bf negative}) if $p^*D-q^*(f_*D)=E$ is effective (resp. is
effective and its support contains the strict transform of the $f$ exceptional divisors).
If $\map X.U.$ and $\map Y.U.$ are projective morphisms, $f\colon\rmap X.Y.$ a birational
contraction over $U$ and $(X,B)$ is a log canonical pair (resp. a divisorially log
terminal $\mathbb Q$-factorial pair) such that $f$ is $(K_X+B)$ non-positive
(resp. $(K_X+B)$-negative) and $K_Y+f_*B$ is nef over $U$ (resp. $K_Y+f_*B$ is nef over
$U$ and $Y$ is $\mathbb{Q}$-factorial), then $f$ is a {\bf weak log canonical model}
(resp. a {\bf minimal model}) of $K_X+B$ over $U$.  If $f\colon\rmap X.Y.$ is a minimal
model of $K_X+B$ such that $K_Y+f_*B$ is semi-ample over $U$, then we say that $f $ is a
{\bf good minimal model} of $K_X+B$ over $U$.  Recall that if $\pi\colon\map X.U.$ is a
projective morphism and $D$ is a $\mathbb{R}$-Cartier divisor on $X$, then $D$ is {\bf
  semi-ample over $U$} if and only if there exists a projective morphism
$g\colon\map X.W.$ over $U$ and an $\mathbb{R}$-divisor $A$ on $W$ which is ample over $U$
such that $g^*A\sim _\mathbb{R} D$.

If $D$ is an $\mathbb R$-divisor on a normal projective variety $X$, then $\phi _D$
denotes the rational map induced by the linear series $|\lfloor D \rfloor |$ and
\[
H^0(X,\ring X.(D))=H^0(X,\ring X.(\lfloor D \rfloor )).
\] 
If $\phi_{D}$ induces a birational map, then we say that $|D|$ is birational.

\subsection{Volumes}

If $X$ is a normal projective variety, $D$ is an $\mathbb{R}$-divisor and $n=\dim X$, then
we define the {\bf volume} of $D$ by 
\[  
\vol(X,D)=\lim \frac {n!h ^0(X,mD)}{m^n}.
\] 
Note that if $D$ is nef, then $\vol (X,D)=D^n$. By definition $D$ is {\bf big} if
$\vol(X,D)>0$.  It is well known that if $D$ is big then $D\sim _\mathbb{R} A+E$ where
$E\geq 0$ and $A$ is ample.  Note that the volume only depends on $[D]\in N^1(X)$, so that
if $D\equiv D'$, then $\vol(X,D)=\vol(X,D')$.  The induced function
$\vol\colon\map N^1(X).\mathbb{R}.$ is continuous \cite[2.2.45]{Lazarsfeld04a}.
\begin{lemma}\label{l-volumes} Let $f\colon\map X.W.$ and $g\colon\map Y.X.$ be birational
morphisms of normal projective varieties and $D$ be an $\mathbb R$-divisor on $X$.  Then
\begin{enumerate}
\item $\vol (W,f_*D)\geq \vol (X,D)$.
\item If $D$ is $\mathbb R$-Cartier and $G$ is an $\mathbb R$-divisor on $Y$ such that
$G-g^*D\geq 0$ is effective and $g$-exceptional, then $\vol(Y,G)=\vol(X,D)$. In particular
if $(X,B)$ is a projective log canonical pair and $f\colon\map Y.X.$ a birational
morphism, then 
\[
\vol(X,K_X+B)=\vol (Y,K_Y+\mathbf{L} _{B,Y})=\vol (Y,K_Y+\mathbf{M} _{B,Y}).
\]
\item If $D\geq 0$, $(W,f_* D)$ has simple normal crossings, and 
$L=\mathbf{L} _{f_*D,X}$, then 
\[
\vol (X,K_X+D)=\vol (X,K_X+D\wedge L).
\]
\item If $(X,B)$ is a log canonical pair and $\rmap X.X'.$ is a birational contraction of
normal projective varieties, then
\[
\vol (X',K_{X'}+\mathbf{M}_{B,X'})\geq \vol (X,K_X+B).
\] 
If moreover $\map X.W.$ and $\map X'.W.$ are morphisms and the centre of every divisor
in the support of $B\wedge \mathbf{L} _{f_*B,X}$ is a divisor on $X'$, then we have equality
\[
\vol (X',K_{X'}+\mathbf{M}_{B,X'})=\vol(X,K_X+B).
\]
\end{enumerate}
\end{lemma}
\begin{proof} If $H\sim mD$, then $f_*H\sim mf_*D$ and so
$h^0(X,\ring X.(mD))\leq h^0(W,\ring W.(mf_*D))$ and (1) follows easily.

(2) follows since $H^0(X,\ring X.(mD))\cong H^0(Y,\ring Y.(mG))$.

To see (3), notice that the inclusion 
\[
H^0(X,\ring X.(m(K_X+D)))\supset H^0(X,\ring X.(m(K_X+D\wedge L)))
\] 
is clear.  We have
\[
K_X+L=f^*(K_W+f_*D)+E,
\] 
where $E\geq 0$ and $L\wedge E=0$.  Now observe that
\begin{align*} 
H^0(X,\ring X.(m(K_X+D))) &\subset f^*H^0(W,\ring W.(m(K_W+f_*D)))\\
                          &=H^0\big(X,\ring X.(m(K_X+L))\big),
\end{align*} 
where we have already demonstrated the inclusion holds and the equality follows as
$E\geq 0$ is exceptional.  But then every section of $H^0(X,\ring X.(m(K_X+D)))$ vanishes
along $mD-mD\wedge mL$ and (3) follows.

To see (4), let $\map X''.X.$ be a resolution of the indeterminacies of $\rmap X.X'.$ so
that $\map X''.X'.$ is also a morphism of normal projective varieties.  Then by (2) and
(1), it follows that
\[
\vol (X,K_X+B)=\vol (X'',K_{X''}+\mathbf{M} _{B,X''})\leq \vol (X',K_{X'}+\mathbf{M} _{B,X'}).
\]
Suppose now that the centre of every divisor in the support of
$B\wedge \mathbf{L} _{f_*B,X}$ is a divisor on $X'$ and let $B'=\mathbf{M}_{B,X'}$.  It is
easy to see that
\[
\mathbf{M} _{B',X''}\wedge \mathbf{L}_{f'_*B',X''}=\mathbf{M} _{B,X''}\wedge \mathbf{L}_{f_*B,X''}
\]
and so by (2) and (3) we have
\begin{align*}
\vol (X,K_X+B)&=\vol (X'',K_{X''}+\mathbf{M}_{B,X''}\wedge \mathbf{L}_{f_*B,X''})\\
              &=\vol (X'',K_{X''}+\mathbf{M}_{B',X''}\wedge \mathbf{L}_{f'_*B',X''})\\
              &=\vol (X',K_{X'}+B').\qedhere
\end{align*}
\end{proof}

\subsection{Non Kawamata log terminal centres}

Here we collect several useful facts about non Kawamata log terminal centres.
\begin{proposition}\label{p-mincentre} Let $(X,B)$ be a log canonical pair and $(X,B_0)$ 
a Kawamata log terminal pair. 
\begin{enumerate}
\item If $W_1$ and $W_2$ are non Kawamata log terminal centres of $(X,B)$ and $W$ is an
irreducible component of $W_1\cap W_2$, then $W$ is a non Kawamata log terminal centre of
$(X,B)$.  In particular if $x\in X$ is a point such that $(X,B)$ is not Kawamata log
terminal in any neighbourhood of $x\in X$, then there is a minimal non Kawamata log
terminal centre $W$ of $(X,B)$ containing $x$.
\item Every minimal non Kawamata log terminal centre $W$ of $(X,B)$ is normal.
\item If $W$ is a minimal non Kawamata log terminal centre of $(X,B)$, then there exists a
divisor $B'\geq 0$ such that for any $0<t<1$, $W$ is the only non Kawamata log terminal
centre of $(X,tB+(1-t)B')$ and there is a unique non Kawamata log terminal place $E$ of
$(X,tB+(1-t)B')$.
\end{enumerate}
\end{proposition}
\begin{proof} For (1-2) see \cite{Kawamata97}. (3) follows from
\cite[8.7.1]{Kollar07}. \end{proof}

\begin{lemma}\label{l-lcc} Let $(X,B)$ be an $n$-dimensional projective log pair and $D$
a big divisor on $X$ such that $\vol (X,D)> (2n)^n$.  Then there exists a family
$\map V.T.$ of subvarieties of $X$ such that if $x$, $y$ are two general points of $X$,
then, possibly switching $x$ and $y$, we may find a divisor $0\leq D_t\sim _\mathbb{R} D$
such that $(X,B+D_t)$ is not Kawamata log terminal at both $x$ and $y$, $(X,B+D_t)$ is log
canonical at $x$ and there is a unique non Kawamata log terminal place of $(X,B+D_t)$ with
centre $V_t$ containing $x$.
\end{lemma}
\begin{proof} Since
\[
h^0(X,\ring X.(kD))=\frac{\vol (X,D)}{n!}\cdot k^n+O(k^{n-1})
\]
and vanishing at a smooth point $x\in X$ to order $l$ imposes
\[
\binom {n+l}{l}=\frac{l^n}{n!}+O(l^{n-1})
\] 
conditions, one sees that for any $s\gg 0$ there is a divisor
$0\leq G_x\sim _\mathbb{R} sD$ such that $\mult_x(G)>2ns$.  Let
\[
\lambda =\sup \{\, l>0 \,|\, \text{$(X,B+l(G_x+G_y))$ is log canonical at one of $x$ or $y$}\,\} < \frac 1{2s}.
\]
If $D'=\lambda(G_x+G_y)+(1-2\lambda s)D$ then $(X,B+D')$ is not Kawamata log terminal at
$x$ and $y$.  Possibly switching $x$ and $y$ we may assume that $(X,B+D')$ is log
canonical in a neighbourhood of $x$.  Perturbing $D'$ we may assume that there is a unique
non Kawamata log terminal place for $(X,B+D')$ whose centre $V$ contains $x$ (see
Proposition \ref{p-mincentre}).  The result now follows using the Hilbert scheme.
\end{proof}

\subsection{Minimal models}

\begin{theorem}[\cite{BCHM10}] Let $(X,B)$ be a $\mathbb{Q}$-factorial Kawamata log
terminal pair and $\pi\colon\map X.U.$ be a projective morphism such that either $B$ or
$K_X+B$ is big over $U$ (respectively $K_X+B$ is not pseudo-effective over $U$) , then
there is a good minimal model $\rmap X.X'.$ of $K_X+B$ over $U$ (respectively a Mori fibre
space $\map X'.Z.$) which is given by a finite sequence of flips and divisorial
contractions for the $K_X+B$ minimal model program with scaling of an ample divisor over
$U$.
\end{theorem}

\begin{theorem}\label{t-model} Let $(X,B)$ be a log pair, then
\begin{enumerate}
\item There is a proper birational morphism $\nu\colon\map X'.X.$ such that $X'$ is
$\mathbb{Q}$-factorial,
\[
K_{X'}+\nu ^{-1}_*B+\Exc(\nu)=\nu ^*(K_X+B)+E
\] 
where $E\leq 0$ and $({X'},\nu ^{-1}_*B+\Exc(\nu))$ is divisorially log terminal.
\item If $(X,B)$ is Kawamata log terminal, then there exists a {\bf $\mathbb{Q}$-factorial
  modification}, that is, a small proper birational morphism $\nu\colon\map X'.X.$ such
that $X'$ is $\mathbb{Q}$-factorial.
\item If $(X,B)$ is $\mathbb{Q}$-factorial and log canonical and $W\subset \Supp( B)$
is a minimal non-Kawamata log terminal centre, then there exists a proper birational
morphism $\nu\colon\map X'.X.$ such that $\rho (X'/X)=1$, $\Exc(\nu )=E$ is an irreducible
divisor and
\[
K_{X'}+\nu ^{-1}_*B+E=\nu ^*(K_X+B).
\]
\end{enumerate}
\end{theorem}\begin{proof} (1) is \cite[3.3.1]{HMX14}.  (2) is an easy consequence of (1)
and (3) follows from \cite[1.4.3]{BCHM10}.
\end{proof}

\begin{proposition} \label{p-MF} Let $(X,B)$ be an $n$-dimensional $\mathbb{Q}$-factorial
divisorially log terminal pair, $0\ne S\leq \lfloor B\rfloor$ and $\pi \colon\map X.U.$ a
projective morphism to a smooth variety.  Let $0\in U$ be a closed point and
$r\in \mathbb{N}$ a positive integer such that $(X_0,B_0)$ is log canonical, $K_{X_0}+B_0$
is nef and $r(K_{X_0}+B_0)$ is Cartier.  Fix $\epsilon <\frac 1{2nr+1}$.

If $K_X+B-\epsilon S$ is not pseudo-effective, then we may run $f\colon\rmap X.Y.$ the
$(K_X+B-\epsilon S)$ minimal model program over $U$ such that
\begin{enumerate}
\item each step is $K_X+B$ trivial over a neighbourhood of $0\in U$,
\item there is a Mori fibre space $\psi\colon \map Y.Z.$ such that $f_*S$ dominates $Z$
and $K_{X'}+f_*B\sim _\mathbb{R} \psi ^*L$ for some $\mathbb{R}$-divisor $L$ on
$Z$.
\end{enumerate}
\end{proposition}
\begin{proof}\cite[5.2]{HMX14}.\end{proof}

\subsection{DCC sets}

A set $I\subset \mathbb{R}$ is said to satisfy the {\bf descending chain condition} (DCC)
if any non increasing sequence in $I$ is eventually constant.  Similarly $I$ satisfies the
{\bf ascending chain condition} (ACC) if any non decreasing sequence in $I$ is eventually
constant or equivalently $-I=\{-i|i\in I\}$ satisfies the DCC.  The {\bf derived set} of
$I$ is defined by
\[
D(I)=\{\, \frac{r-1+i_i+\ldots +i_p}r \,|\, r\in \mathbb{N} ,\ i_j\in I \,\}.
\]
Note that $I$ satisfies the DCC if and only if $D(I)$ satisfies the DCC.

\begin{lemma}\label{l-5.2} Let $I\subset [0,1]$ be a DCC set and $J_0\subset [0,1]$ a
finite set, then
\[
I_1:=\{\, i\in I\,|\, \text{$\frac {m-1+f+ki}m\in J_0$, where $k$, $m\in \mathbb{N}$ and  $f\in D(I)$}\,\}
\]
is a finite set.
\end{lemma}
\begin{proof} See \cite[5.2]{HMX14}.
\end{proof}

\subsection{Good minimal models}

We will need the following results from \cite[1.2, 1.4]{HMX14}.

\begin{theorem}\label{t-gmm} Let $(X,B)$ be a log pair and $\pi\colon\map X.U.$ a projective
morphism to a smooth affine variety such that $\coeff(B)\subset (0,1]\cap \mathbb{Q}$ and
$(X,B)$ is log smooth over $U$.  Suppose that there is a point $0\subset U$ such that
$(X_0, B_0)$ has a good minimal model.  Then $(X,B)$ has a good minimal model over $U$ and
every fibre has a good minimal model.

Furthermore, the relative ample model of $(X,B)$ over $U$ gives the relative ample model
of each fibre.
\end{theorem}
\begin{proof} \cite[1.2, 1.4]{HMX14}.
\end{proof}

\begin{theorem}\label{t-definv} Let $(X,B)$ be a log pair such that $\coeff(B)\subset (0,1]$
and $\pi\colon\map X.U.$ a projective morphism such that $(X,B)$ is log smooth over $U$.
Then $h^0(X_u,\ring X_u.(m(K_{X_u}+B_u)))$ is independent of $u\in U$.  In particular
\[
\map f_*{\ring X.(m(K_X+B))}.{H^0(X_u,\ring X_u.(m(K_{X_u}+B_u)))}.
\] 
is surjective for all $u\in U$.
\end{theorem}
\begin{proof} Notice that
\[
\ring X.(m(K_X+B))=\ring X.(\lfloor m(K_X+B)\rfloor )=\ring X.( m(K_X+B_{\lfloor m\rfloor}))
\]
where $ B_{\lfloor m\rfloor }=\lfloor mB\rfloor/m$.  The statement now follows from
\cite[1.2]{HMX14}.
\end{proof}

\subsection{Log birational boundedness}

We begin with the following easy:
\begin{lemma}\label{l-pb} Let $(X,B)$ be a projective log pair and $D$ a big
$\mathbb{R}$-divisor such that for general points $x$, $y\in X$ there is an
$\mathbb{R}$-divisor $0\leq D'\sim _\mathbb{R} \lambda D$ for some $\lambda <1$ such that
\begin{enumerate}
\item $x$ is an isolated non Kawamata log terminal centre of $(X,B+D')$,
\item $(X,B+D')$ is log canonical in a neighbourhood of $x\in X$, and
\item $(X,B+D')$ is not Kawamata log terminal at $y$.
\end{enumerate}
Then $\phi _{K_X+\lceil D \rceil}$ is birational.
\end{lemma}
\begin{proof} Fix a resolution $\nu\colon\map X'.X.$ of $(X,B)$.  As $x$ and $y$ are
general, $\nu$ is an isomorphism in a neighbourhood of $x$ and $y$.  If
$\phi _{K_{X'}+\lceil \nu ^*D\rceil}$ is birational then so is
$\phi _{K_X+\lceil D \rceil}$.  Therefore, replacing $X$ by $X'$, we may assume that $X$
is smooth and $B=0$.

As $x$ is an isolated non Kawamata log terminal centre, we have
$\mathcal J(D')\cong \frak m _x$ in a neighbourhood of $x\in X$.  By Nadel vanishing (see
\cite[9.4.8]{Lazarsfeld04b}), we have
\[
H^1\big(X,\ring X.(K_X+\lceil D \rceil)\otimes \mathcal J (D')\big)=0
\] 
and so
\[
\map {H^0(X,\ring X.(K_X+\lceil D \rceil))}.{H^0(X,\ring X.(K_X+\lceil D \rceil)\otimes \ring X./\mathcal J (D'))}.
\] 
is surjective. 

But since $\ring X./ \mathcal J (D')$ has a summand isomorphic to
$\ring X./\frak m _x$ and another summand whose support contains $y$, it follows that
we may lift a section of
$H^0(X,\ring X.(K_X+\lceil D \rceil)\otimes \ring X./ \mathcal J (D'))$ not vanishing
at $x$ and vanishing at $y$ to a section of $H^0(X,\ring X.(K_X+\lceil D \rceil))$ and
the assertion is proven.
\end{proof}

\begin{definition}\label{d-bounded} We say that a set of varieties $\mathfrak{X}$ is 
{\bf bounded} (resp.  {\bf birationally bounded}) if there exists a projective morphism
$\map Z.T.$, where $T$ is of finite type, such that for every $X\in \mathfrak{X}$, there
is a closed point $t\in T$ and an isomorphism (resp. a birational map)
$f\colon\map X.Z_t.$.  We say that a set $\mathfrak{D}$ of log pairs is {\bf bounded}
(resp. {\bf log birationally bounded}) if there is a log pair $(Z,D)$, where the
coefficients of $D$ are all one, and there is a projective morphism $\map Z.T.$, where $T$
is of finite type, such that for every pair $(X,B)\in \mathfrak{D}$, there is a closed
point $t\in T$ and a map $f\colon\map Z_t.X.$ inducing an isomorphism
$(X,B_{\rm red})\cong (Z_t,D_t)$ (resp. such that the support of $D_t$ contains the
support of the strict transform of $B$ and any $f$-exceptional divisor).
\end{definition}

\begin{remark} Note that, by a standard Hilbert scheme argument, a set of varieties
$\mathfrak{X}$ (resp. of pairs) is bounded if there exists a constant $C>0$ such that for
each $X\in \mathfrak{X}$ there is a very ample divisor $H$ on $X$ such that
$H^{\dim X}\leq C$ (resp. $H^{\dim X}\leq C$ and $B_{\rm red}\cdot H^{\dim X-1}\leq C$).
\end{remark}

\begin{proposition}\label{p-BB} Fix $n\in \mathbb{N}$, $A>0$ and $\delta >0$.  The set of
projective log canonical pairs $(X,B=\sum b_iB_i)$ such that
\begin{enumerate}
\item $\dim X =n$,
\item  $b_i\geq \delta$, 
\item $|m(K_X+B)|$ is birational, and 
\item $\vol (m(K_X+B))\leq A$
\end{enumerate} 
is log birationally bounded.
\end{proposition}
\begin{proof} We first reduce to the case that the rational map
\[
\phi\colon =\phi _{m(K_X+B)}\colon\rmap X.Z.
\] 
is a birational morphism.  To see this let $\nu \colon\map X'.X.$ be a resolution of the
indeterminacies of $\phi$ and $B'=\nu ^{-1}_*B+\Exc(\nu)$, then $K_{X'}+B'-\nu ^*(K_X+B)$
is an effective exceptional divisor.  In particular
\[
\vol (X',K_{X'}+B')=\vol (X,K_X+B)\leq A
\]
and $\phi _{m(K_{X'}+B')}$ is birational.  Therefore it suffices to show that the pairs
$(X',B')$ are log birationally bounded.  Replacing $(X,B)$ by $(X',B')$ we may assume that
$\phi$ is a morphism.

Let $|\lfloor m(K_X+B)\rfloor |=|M|+E$ where $E$ is the fixed part and $|M|$ is base point
free so that $M=\phi^*H$ for some very ample divisor $H$ on $Z$.  Note that
\[
H^n=\vol(Z,H)\leq \vol(X,m(K_X+B))\leq A
\] 
and hence it suffices to show that $\phi _* (B_{\rm red})\cdot H^{n-1}$ is bounded.  Let
$B_0=\phi ^{-1}_*\phi _* B_{\rm red}$ and $L=2(2n+1)H$.  By Lemma \ref{l-xy}
\[
|K_X+(n+1)\lfloor m(K_X+B)\rfloor |\neq \emptyset
\]
and since $B_0\leq \frac 1 \delta B$ it follows that there is an effective
$\mathbb{R}$-divisor $C$ such that
\begin{equation}\label{e-div}
B_0+C\sim _\mathbb{R} \frac {m(n+1)+1} \delta (K_X+B).
\end{equation} 
We have that
\begin{align*} 
\phi _* B_{\rm red}\cdot L^{n-1}&=B_0\cdot (2(2n+1)M)^{n-1}\\
                              &\leq 2^n\vol(X,K_X+B_0+2(2n+1)M)\\
                              &\leq 2^n\vol(X,K_X+\frac {m(n+1)+1}\delta (K_X+B)+2(2n+1)m(K_X+B))\\
                              &\leq 2^n(1+(\frac {n+2} \delta +4n +2))^n\vol(X,m(K_X+B))\\
                              &\leq  2^n(1+(\frac {n+2} \delta +4n +2))^nA,
\end{align*} 
where the second inequality follows from Lemma \ref{l-volin} and the third from equation
\eqref{e-div}.
\end{proof}

\begin{lemma}\label{l-volin} Let $X$ be an $n$-dimensional normal projective variety, $M$ a 
Cartier divisor such that $|M|$ is base point free and $\phi _M$ is birational.  If
$L=2(2n+1)M$, and $D$ is a reduced divisor, then
\[
D\cdot L ^{n-1}\leq 2^n \vol(X,K_X+D+L).
\]
\end{lemma}
\begin{proof} Let $\nu\colon\map X'.X.$ be a proper birational morphism. Since
\[
\vol(X',K_{X'} +\nu ^{-1}_*D +\nu ^*L)\leq \vol (X,K_X+D+L),
\]
we may assume that $X$ and $D$ are smooth (in particular the components of $D$ are
disjoint).  Let $\phi =\phi _M\colon\map X.Z.$ be the induced birational morphism.  Since
\[
D\cdot L ^{n-1} =\phi ^{-1}_*\phi _* D\cdot L ^{n-1},
\] 
we may assume that no component of $D$ is contracted by $\phi$ and we may replace $D$ by
$\phi ^{-1}_*\phi _* D$.  Thus we can write $M\sim _\mathbb{Q} A+B$ where $A$ is ample,
$B\geq 0$ and $B$ and $D$ have no common components.  In particular $K_X+D+\delta B$ is
divisorially log terminal for some $\delta >0$ and so, by Kawamata-Viehweg vanishing,
\[
H^i(X,\ring X.(K_X+E+pM))=0
\]
for all $i>0$, $p>0$ and reduced divisors $E$ such that $0\leq E\leq D$.  In particular
$H^i(X,\ring D.(K_D+pM|_D))=0$ for all $i>0$, $p>0$.  Therefore there are surjective
homomorphisms
\[
\map{H^0(X,\ring X.(K_X+D+(2n+1)M))}.{H^0(D,\mathcal O_D(K_D+(2n+1)M|_D))}..
\]
By Lemma \ref{l-xy}, $|K_D+(2n+1)M|_D|$ is non-empty and so the general section of
$H^0(X,\ring X.(K_X+D+(2n+1)M))$ does not vanish along any component of $D$. It is
also easy to see that $|2K_X+D+2L|$ is non empty.  Consider the commutative diagram
\[
\begin{diagram}
\ring X.(K_X+mL+D)    &   \rTo    &    \ring D.(K_D+mL|_D) \\
\dTo &  & \dTo \\
\ring X.((2m-1)(K_X+L+D))   &   \rTo    & \ring D.((2m-1)(K_D+L|_D))
\end{diagram}
\]
whose vertical maps are induced by a general section
of 
\[
(m-1)(2K_X+L+2D)=2(m-1)(K_X+(2n+1)M+D).
\]
Since 
\[
|2K_X+2D+L|=|2(K_X+D+(2n+1)M)|
\] 
is non empty and $H^1(X,\ring X.(K_X+mL))=0$, it follows that
\begin{align*} 
&\phantom{\geq {}} h^0(X,\ring X.((2m-1)(K_X+L+D)))-h^0(X,\ring X.((2m-3)(K_X+L+D))) \\
&\geq h^0(X,\ring X.((2m-1)(K_X+L+D)))-h^0(X,\ring X.((2m-2)(K_X+L+D)+K_X+L))\\
&=\dim \im \left(\map {H^0(X,\ring X.((2m-1)(K_X+L+D)))}.{H^0(D,\ring D.((2m-1)(K_D+L|_D)))}.\right) \\
&\geq h^0(D,\ring D.((K_D+mL|_D))).
\end{align*} 
The leading coefficient of $m^n$ in
\[
h^0(X,\ring X.((2m-1)(K_X+L+D)))
\] 
is $2^n\vol (K_X+L+D)/n!$ and, by the vanishing observed above,
\[
h^0(D,\ring D.(K_D+mL|_D))=\chi(D,\ring D.(K_D+mL|_D))
\]
is a polynomial of degree $n-1$ whose leading coefficient is $D\cdot L^{n-1}/(n-1)!$.
Comparing these coefficients, one sees that
\[
\vol (X,K_X+L+D) \cdot \frac {2^n}{n!}\geq \frac {D\cdot L^{n-1}}{n!}.\qedhere
\]
\end{proof}

\begin{lemma}\label{l-xy} Let $X$ be an $n$-dimensional smooth projective variety and $M$
a Cartier divisor such that $|M|$ is base point free and $\phi _{|M|}$ is generically
finite (resp.  birational).  Then there is an open subset $U\subset X$ such that for any
$x\in U$ (resp. $x$, $y\in U$) and any $t\geq n+1$ (resp. $t\geq 2n+1$) there is a section
$H^0(X,\ring X.(K_X+tM))$ not vanishining at $x$ (resp vanishing at $y$ and not vanishing
at $x$).
\end{lemma}
\begin{proof} Let $\phi=\phi _M\colon\map X.Z.$ be the induced morphism so that
$M=\phi^*H$ where $H$ is very ample.  Let $U\subset X$ be the open subset on which $\phi$
is finite (resp. an isomorphism).  We may pick a divisor $G\sim _\mathbb{Q} \lambda H$
such that $\lambda <n+1$ (resp. $\lambda <2n+1$) and $\phi(x)$ (resp. $\phi(x)$ and
$\phi(y)$) are isolated points in the cosupport of $\mathcal{J}(G)$ (this can be achieved
by letting $G= \frac n{n+1}\sum _{i=1}^{n+1}H_i$ where $H_i\in |H|$ are general
hyperplanes containing $\phi (x)$ (resp. letting
$G= \frac n{n+1}(H_0+\sum _{i=1}^{n}(H_i+H'_i))$ where $H_i\in |H|$ are general
hyperplanes containing $\phi (x)$, $H'_i\in |H|$ are general hyperplanes containing
$\phi (y)$ and $H_0$ is a general hyperplane containing $\phi (x)$ and $\phi (y)$)).

By Nadel vanishing, $H^1(X,\ring X.(K_X+tM)\otimes \mathcal J (\phi ^* G))=0$ and so there
is a surjection
\[
\map {H^0(X,\ring X.(K_X+tM))}.{H^0(X,\ring X.(K_X+tM)\otimes \ring X./\mathcal J (\phi ^* G))}..
\]
Since $x$ (resp. $x$, $y$) is an isolated component of the support of
$\ring X./\mathcal J (\phi^*G)$, it follows that there is a section
$H^0(X,\ring X.(K_X+tM))$ not vanishining at $x$ (resp. vanishing at $y$ and not vanishing
at $x$).
\end{proof}

\section{Pairs with hyperstandard coefficients}

\subsection{The DCC for volumes of log birationally bounded pairs}

In this section, we will show that for pairs which are log birationally bounded, if the
coefficients satisfy the DCC, then the volumes satisfy the DCC.

Let $(Z,D)$ be a simple normal crossings pair with the obvious toroidal structure.  Let
$\mathbf{B}$ be an effective $b$-divisor.  If $\llist Y.m.$ are finitely many toroidal
models over $Z$, and $\map Z'.Z.$ is a proper birational morphism factoring through each
$Y_i$, then we define the corresponding {\bf cut} $(Z',\mathbf{B} ')$ of $(Z,\mathbf{B} )$
to be
\[
\mathbf{B}'=\mathbf{B}\wedge\mathbf{M}_{\Theta}\qquad \text{where}\qquad B_i=\mathbf{L} _{\mathbf{B}_{Y_i},Z'}\quad \text{and}\quad \Theta=\wedge^m_{i=1} B_i.
\]
Notice that the coefficients of $\mathbf{B} '$ belong to the set
$I'=I\cup \coeff(\mathbf{B}' _{Z'})$, which also satisfies the DCC.  If $Z$ is projective
then it follows easily from Lemma \ref{l-volumes} that
$\vol(Z',K_{Z'}+\mathbf{B} '_{Z'})=\vol(Z',K_{Z'}+\mathbf{B} _{Z'})$.  We say that
$(Z',\mathbf{B} ')$ is a {\bf reduction} of $(Z,\mathbf{B})$ if it is obtained by a finite
sequence $(Z_i,\mathbf{B}_i)$, $i=0, 1,\ldots, k$, where
$(Z_0,\mathbf{B} _0)=(Z,\mathbf{B})$, $ (Z_i ,\mathbf{B} _i)$ is a cut of
$(Z_{i-1} ,\mathbf{B} _{i-1})$ for $i=1,\ldots , k$ and
$ (Z_k ,\mathbf{B} _k)=(Z',\mathbf{B} ')$.

\begin{lemma}\label{l-stable} Let $(Z,D)$ be a simple normal crossings pair and
$\mathbf{B}$ a $b$-divisor whose coefficients are in a DCC set contained in $[0,1]$ such
that $\Supp (\mathbf{B}_Z)\subset \Supp(D)$, then there exists a reduction
$(Z',\mathbf{B} ')$ of $(Z,\mathbf{B})$ that
\[
\mathbf{B}'\ge \mathbf{L}_{\mathbf{B}'_{Z'}}.
\]
\end{lemma}
\begin{proof} If $\mathbf{B} \ge \mathbf{L} _{\mathbf{B} _Z}$, then there is nothing to
prove. Suppose now that for any divisorial valuation $\nu $ such that
$\mathbf{B} (\nu) < \mathbf{L} _{\mathbf{B} _Z}(\nu)$, the centre of $\nu$ is not
contained in any strata of $\lfloor \mathbf{B}_Z \rfloor $.  Let $\map Z'.Z.$ be a finite
sequence of blow ups along strata of $\{ \mathbf{B} _Z\}$ such that
$(Z',\{\mathbf{L} _{\mathbf{B} _Z,Z'}\})$ is terminal and let
$\mathbf{B}'=\mathbf{B} \wedge \mathbf{M} _{\mathbf{L} _{\mathbf{B} _Z,Z'}}$, that is,
$\mathbf{B}'$ is the cut of $\mathbf{B}$ with respect to $\map Z'.Z.$.  Let $\nu$ be a
valuation such that $\mathbf{B} ' (\nu) < \mathbf{L} _{\mathbf{B} ' _{Z'}}(\nu)$.  Since
$\mathbf{B}' _{Z'}= \mathbf{L} _{\mathbf{B} '_{Z'}, Z'}$, the centre of $\nu$ is not a
divisor on $Z'$.  But then
\[
\mathbf{B} (\nu )=\mathbf{B}'(\nu)< \mathbf{L} _{\mathbf{B} ' _{Z'}}(\nu)\leq \mathbf{L}_{\mathbf{B} _Z}(\nu)
\]
and so the centre of $\nu$ is not contained in any strata of
$\lfloor \mathbf{B}_{Z'} ' \rfloor $ (as the strata of
$\lfloor \mathbf{B}_{Z'} ' \rfloor $ map to strata of $\lfloor \mathbf{B}_Z \rfloor $).
But then, since $(Z',\{ \mathbf{B} '_{Z'}\})$ is terminal
$\mathbf{L} _{\mathbf{B} ' _{Z'}}(\nu)=0$ which is impossible as $\mathbf{B}\geq 0$.

We may therefore assume that there is a divisorial valuation $\nu$ with centre contained
in a stratum of $\lfloor \mathbf{B}_Z\rfloor$.  Let $k=k(Z,\mathbf{B} )$ be the maximal
codimension of such a stratum. We will prove the statement by induction on $k\geq 1$. It
suffices to show that there is a cut $(Z',\mathbf{B} ')$ of $(Z,\mathbf{B})$ such that
$k(Z',\mathbf{B} ')<k(Z,\mathbf{B} )$.  Since $\lfloor \mathbf{B}_Z\rfloor$ has finitely
many strata, we may work locally around each stratum.  We may therefore assume that
$Z=\mathbb{C}^n $ and
\[
\mathbf{B} _Z=E_1+\ldots +E_k+\sum _{i=k+1}^na_i E_i
\]
where $0\leq a_i<1$ and the $E_i$ are the coordinate hyperplanes.  The divisors
$\llist E.n.$ correspond to vectors $\llist e.n.$ and the valuations we consider
correspond to the non zero integral prime vectors with non negative coefficients.  Let $E$
be a toric valuation corresponding to a vector $\sum^n_{j=1}b_je_j$ for
$b_j\in \mathbb{N}$.  By standard toric geometry the coefficient of $E$ in
$\mathbf{L}_{\mathbf{B}_Z}$ is
\begin{eqnarray}\label{logdis}
\left( 1-\sum^n_{j=k+1}(1-a_j)b_j \right) \vee 0.
\end{eqnarray}

By equation \ref{logdis}, there is a finite set $V_0\subset \mathbb{N} ^{n-k}$, such that
if $(b_{k+1},\ldots ,b_n)\notin V_0$, then any divisor $E$ corresponding to a vector
\[
(*,...,*, b_{k+1},\ldots ,b_n)
\] 
where the first $k$ entries are arbitrary, satisfies $\mathbf{L}_{\mathbf{B}_Z}(E)= 0$.

In what follows below, by abuse of notation, we will use $\sigma$ to denote both an
integral prime vector in $\mathbb{N}^n$ and the corresponding toric valuation.  Fix
$v\in V_0$, among all valuations of the form $\sigma =(*,...,*, v)$, we consider
$\sigma =\sigma (v)$ such that $\mathbf{B}(\sigma )$ is minimal (this is possible since
the coefficients of $\mathbf{B}$ belong to a DCC set).  We pick a toroidal log resolution
$\map Z'.Z.$ such that for any $v\in V_0$, and any $\sigma =\sigma (v)$ as above, the
induced rational map $\map Z'.Y_\sigma.$ is a morphism (where $\map Y_\sigma.Z.$ is the
toroidal morphism with exceptional divisor $E_\sigma$ such that $\rho(Y_{\sigma}/Z)=1$).
Let $(Z',\mathbf{B}')$ be the cut corresponding to $\{ Y_{\sigma}\}_{v\in V_0}$.

We claim that $k(Z',\mathbf{B}')<k(Z,\mathbf{B} )=k$.  Suppose to the contrary that there
exists a stratum of $\mathbf{B}'_{Z'}$ of codimension $k$ containing the centre of a
valuation $\nu$ such that $\mathbf{B}'(\nu )<\mathbf{L} _{\mathbf{B}' _{Z'}}(\nu )$.
Clearly $\nu$ is exceptional over $Z'$ and $\nu$ is a toric valuation corresponding to a
vector $\tau =(*,\ldots , *, b_{k+1},\ldots , b_n)$ for some
$v=(b_{k+1},\ldots , b_n)\in V_0$.  Let $\sigma =\sigma (v)$ be the valuation defined
above.  Then $\map Z'.Z.$ factors through $Y_\sigma$.  The toric morphism
$\map Y_\sigma.Z.$ corresponds to subdividing the cone given by the basis vectors
$\llist e.n.$ in to $m\leq n$ cones spanned by $\sigma$ and
$e_1,\ldots , e_{l-1}, e_{l+1},\ldots , e_n$.  Since $\tau$ belongs to one of these cones,
we may write
\[
\tau =\lambda \sigma +\sum _{i\ne l}\lambda _i e_i,\qquad \lambda , \lambda _i\in \mathbb{Q} _{\geq 0}.
\]
Since $\lfloor \mathbf{B} '_{Z'}\rfloor$ has a codimension $k$ stratum containing the
centre of $\tau$, the same is true for $(\mathbf{L}_{\mathbf{B}})_{Y_\sigma}$.  Since
$\mathbf{L} _{\mathbf{B}_Z} (\sigma )<1$, it follows that $b_l\neq 0$ for some
$k+1\leq l\leq n$.  Therefore $\tau$ belongs to a cone spanned by $\sigma$ and
$\{e_i\}_{i\ne l}$ where $k+1\leq l\leq n$.  Since the last $n-k$ coordinates of $\sigma$
and $\tau$ are identical, it follows that $\lambda =1$ and hence $\tau \geq \sigma$ (in
the sense that $\tau -\sigma =\sum _{i\ne l}\lambda _i e_i$ where $\lambda _i\geq 0$).

It then follows that 
\begin{align*} 
\mathbf{L} _{\mathbf{B} '_{Z'}}(\tau)&\leq {\mathbf{L}_{\mathbf{B}_{Y_\sigma}}}(\tau)\\
                                  &\leq {\mathbf{L}_{\mathbf{B} _{Y_\sigma}}}(\sigma)\\
                                  &\leq \mathbf{B}(\sigma )\\
                                  &\leq \mathbf{B}(\tau)\\
                                  &=\mathbf{B}'(\tau ),
\end{align*} 
where the first inequality follows from the definition of $\mathbf{B} '$, the second as
$\tau \leq \sigma$, the third as $\sigma$ is a divisor on $\mathbf{B} _{Y_\sigma}$, the
fourth by our choice of $\sigma$, and the fifth by the definition of $\mathbf{B} '$ and
the fact that $\tau$ is exceptional over $Z'$.
\end{proof}

The following theorem is \cite[5.1]{HMX10}. 
\begin{theorem}\label{t-voldcc} Fix a set $I\subset [0,1]$ which satisfies the DCC and $(Z,D)$ 
a simple normal crossing pair where $D$ is reduced.  Consider the set $\mathfrak{D}$ of
all projective simple normal crossings pairs $(X,B)$ such that $\coeff(B)\subset I$,
$f\colon\map X.Z.$ is a birational morphism and $f_*B\leq D$.  Then the set
\[
\{\,\vol(X,K_X+B)\,|\,(X,B)\in \mathfrak{D}\,\}
\]
satisfies the DCC.
\end{theorem}
\begin{proof} Suppose that $(X_i,B_i)$ is an infinite sequence of pairs in $\mathfrak{D}$
such that $\vol(X_i,K_{X_i}+B_i)$ is a strictly decreasing sequence.  Let
$f_i\colon\map X_i.Z.$ be the induced morphisms such that $\Supp(f_{i*}B_i)\subset D$.  By
Lemma \ref{l-volumes}, we know that
\[
\vol(X_i,K_{X_i}+B_i)\leq \vol (Z,K_Z+f_{i*}B_i)
\]
and if this inequality is strict then
\[
\mult_E(K_{X_i}+B_i)<\mult_Ef_i^*(K_Z+f_{i*}B_i)
\]
for some divisor $E$ contained in $\Exc(f_i)$.  In this case $E$ must define a toroidal
valuation with respect to the toroidal structure associated to $(Z,D)$.  Let
$f_i'\colon\map X_i'.Z.$ be a toroidal morphism such that if $E$ is a divisor on $X_i$
corresponding to a toroidal valuation for $(Z,D)$, then $E$ is a divisor on $X_i'$.  Let
$B_i'$ be the strict transform of $B_i$ plus the $\rmap X'_i.X_i.$ exceptional divisors.
By Lemma \ref{l-volumes}, we have
\[
\vol(X_i,K_{X_i}+B_i)=\vol (X'_i,K_{X'_i}+B'_i)
\]
and the coefficients of $B_i'$ are contained in $I\cup \{1\}$.  Replacing $(X_i,B_i)$ by
$(X_i',B_i')$, we may assume that $f_i\colon\map (X_i,B_i).(Z,D).$ are toroidal morphisms.

For each pair $(X_i,B_i)$, consider the $b$-divisors $\mathbf{M} _{B_i}$.  Since there are
only countably many toroidal valuations over $(X,D)$, by a standard diagonalisation
argument, after passing to a subsequence, we may assume that for any toroidal divisor $E$
over $X$, the sequence $\mathbf{M}_{B_i}(E)$ is eventually non decreasing.  Note that for
any non-toroidal divisor $E$ exceptional over $Z$, we have
$\mathbf{M}_{B_i}(E)=1$.  Therefore the limit $\lim \mathbf{M} _{B_i}(E)$ exists for any
divisor $E$ over $Z$.  We let $\mathbf{B}=\lim \mathbf{M} _{B_i}$ so that
$\mathbf{B}(E)=\lim \mathbf{M} _{B_i}(E)$ for any divisor $E$ over $Z$.

By Lemma \ref{l-stable} there is a reduction $(\mathbf{B}',Z')$ of $(\mathbf{B}, Z)$
such that $\mathbf{B} '\geq \mathbf{L} _{\mathbf{B} '_{Z'}}$ and
\[
\vol(Z',K_{Z'}+\mathbf{B}'_{Z'})=\vol (Z',K_{Z'}+\mathbf{B}_{Z'}).
\]

Since for any divisor $E$ over $Z$, the sequence $\mathbf{B}_i(E)$, $i=1,\dots $ is
eventually non decreasing, we may assume that $\mathbf{M} _{B_i,Z'}\leq \mathbf{B}_{Z'}$.
Therefore, we have
\begin{align*} 
\vol (X_i,K_{X_i}+B_i) &\leq \vol (Z',K_{Z'}+\mathbf{M} _{B_i,Z'})\\
                      &\leq \vol (Z',K_{Z'}+\mathbf{B}_{Z'}) \\
                      &=\vol(Z',K_{Z'}+\mathbf{B}'_{Z'}).
\end{align*} 

If $\vol(X_i,K_{X_i}+B_i)$, $i=1, \dots$ is strictly decreasing then there exists a
constant $\epsilon>0$, such that for any $j>{i+1}$,
\begin{align*} 
\vol (X_j,K_{X_j}+B_j)&<  \vol (X_{i+1},K_{X_{i+1}}+B_{i+1})\\
                     &\leq \vol (Z',K_{Z'}+(1-\epsilon)\mathbf{B}'_{Z'}).
\end{align*} 

Let $\map Z''.Z'.$ be a toroidal morphism which extracts all divisors $E$ with
$a(E, {Z'},(1-\epsilon)\mathbf{B}'_{Z'} )<0$.  Then
\[
\mathbf{M}_{B_j,Z''} \geq  \mathbf{L}_{(1-\epsilon)\mathbf{B}'_{Z'},Z''}
\]
for all $j\gg 0$, which then implies
\[
\mathbf{M}_{B_j,X_j} \geq \mathbf{L}_{(1-\epsilon)\mathbf{B}'_{Z'},X_j}
\]
as for any exceptional divisor $E$ of $X_j/Z''$, we know that
$\mathbf{L}_{(1-\epsilon)\mathbf{B}'_{Z'}}(E)=0$.  But then
\begin{align*} 
\vol(X_j,K_{X_j}+B_j)&\geq \vol (X_j,K_{X_j}+ \mathbf{L}_{(1-\epsilon)\mathbf{B}'_{Z'},X_j} )\\
                    &=\vol (Z'',K_{Z''}+\mathbf{L}_{(1-\epsilon)\mathbf{B}'_{Z'},Z''})
\end{align*} 
which is the required contradiction.
\end{proof}

\begin{corollary}\label{c-voldcc}\cite[1.9]{HMX10} Fix a set $I\subset [0,1]$ which satisfies the DCC  
and a log birationally bounded set $\mathfrak{B}_0$ of log canonical pairs $(X,B)$ such
that $\coeff(B)\subset I$.  Then the set
\[
\{\, \vol (X,K_X+B) \,|\, (X,B)\in \mathfrak{B}_0 \,\}
\]
satisfies the DCC.
\end{corollary}
\begin{proof} We may assume that $1\in I$.  Since $\mathfrak{B}_0$ is log birationally
bounded, there exists a projective morphism $\map Z.T.$ where $T$ is of finite type and a
log pair $(Z,D)$ such that for any $(X,B)\in \mathfrak{B}_0$, there is a birational
morphism $f\colon\rmap X.Z_t.$ such that the support of $D_t$ contains the support of
$f_*B$ and the $\rmap Z_t.X.$ exceptional divisors.  Let $\nu\colon\map X'.X.$ be a
resolution such that $\map X'.Z_t.$ is a morphism.  If $B'=\mathbf{M} _{B,X'}$, then
$\vol(X',K_{X'}+B')=\vol (X,K_X+B)$ (cf. Lemma \ref{l-volumes}).  Replacing $(X,B)$ by
$(X',B')$ we may assume that $f$ is a morphism.  By a standard argument, after replacing
$T$ by a finite cover and $X$ by the corresponding fibre product, we may assume that $T$
is smooth (and possibly reducible), $(Z,D)$ is log smooth over $T$, the strata of $(Z,D)$
are geometrically irreducible over $T$.

For any birational morphism $f\colon\map X.Z_t.$ as above, consider the finite set
$\mathcal E$ of divisors $E$ on $X$ such that $a_E(Z_t,f_* B)<0$.  Since $(Z_t,D_t)$ is
log smooth, there exists a finite sequence of blow ups along strata of $\mathbf{M}_{D_t}$
say $\map X'.Z_t.$ such that the divisors in $\mathcal E$ are not $\rmap X.X'.$
exceptional.  Let $p\colon\map W.X.$ and $q\colon\map W.X'.$ be a common resolution, then
by Lemma \ref{l-volumes}
\[
\vol (X,K_X+B)=\vol (W,K_{W}+\mathbf{M}_{B,W})=\vol (X',K_{X'}+B')
\]
where $B'=\mathbf{M} _{B,X'}$.  Hence, replacing $(X,B)$ by $(X',B')$ we may assume that
each $f\colon\map X.Z_t.$ is induced by a finite sequence of blow ups along strata of
$(Z_t,D_t)$.  Notice that since $(Z,D)$ is log smooth over $T$, there is a sequence of
blow ups along strata of $(Z,D)$ say $\map Z'.Z.$ such that $Z'_t \cong X'$.  Let $\Phi$
be the divisor supported on the strict transform of $D$ and the $\map Z'.Z.$ exceptional
divisors such that $\Phi _t=B'$.  Fix a closed point $0\in T$, then by Theorem
\ref{t-definv},
\[
\vol (X,K_X+B)=\vol (Z'_0,K_{Z'_0}+\Phi_0).
\]
By Theorem \ref{t-voldcc}, the set of these volumes satisfies the DCC.
\end{proof}

\begin{theorem}\label{t-volddc1}\cite[3.5.2]{HMX14} Fix $n\in \mathbb{N}$, $M>0$, and a  set 
$I\subset [0,1]$ which satisfies the DCC.  Suppose that $\mathfrak{B}_0$ is a set of log
canonical pairs $(X,B)$ such that
\begin{enumerate}
\item $X$ is projective of dimension $n$,
\item $\coeff(B)\subset I$, and
\item there exists an integer $k>0$ such that $\phi _{k(K_X+B)}$ is birational and
$\vol(X,k(K_X+B))\leq M$.
\end{enumerate} 
Then the set
\[
\{\, \vol(X,K_X+B) \,|\, (X,B)\in \mathfrak{B}_0 \,\}
\]
satisfies the DCC.
\end{theorem}
\begin{proof} Proposition \ref{p-BB} implies that $\mathfrak{B}_0$ is log birationally
bounded and so the result follows from Corollary \ref{c-voldcc}.
\end{proof}

\subsection{Adjunction}

In this section, we discuss various versions of adjunction. The main new result is Theorem
\ref{t-4.2}, which is an adjunction for pairs with hyperstandard coefficients.  This is a
key result for doing induction on the dimension.  We remark that the proof of Theorem
\ref{t-4.2} is different from the one in \cite[4.2]{HMX14}, as the argument presented here
does not need the ACC for log canonical thresholds.  On the other hand the argument only
works for hyperstandard coefficients.

\begin{theorem}[Shokurov log adjunction]\label{tSLA} Let $(X,S+B)$ be a log canonical pair, 
$B=\sum b_iB_i$ and $S$ a prime divisor with normalisation $\nu\colon\map S^\nu.S.$.  Then
\[
(K_X+S+B)|_{S^\nu}=K_{S^\nu}+\Diff_{S^\nu}(B)=K_{S^\nu}+\Diff_{S^\nu}(0)+B|_{S^\nu},
\]
where the coefficients of $\Diff_{S^\nu}(B)$ are of the form
\[
\frac {r-1+\sum n_ib_i}{r}\qquad \text{for some} \qquad n_i\in \mathbb{N}
\]
and $r$ is the index of the corresponding codimension $2$ point $P\in X$.  In particular
if $\coeff(B)\subset I$, then $\coeff(\Diff _{S^\nu}(B))\subset D(I)$.
\end{theorem}

We have the following easy consequence.
\begin{lemma}\label{l-diff}  Let $(X,S+B)$ be a log canonical pair, $B=\sum b_iB_i$ 
an effective $\mathbb{R}$-Cartier divisor and $S$ a prime divisor with
normalisation $\nu \colon\map S^\nu.S.$.  Then for any $0\leq \lambda \leq 1$ we have
\[
\Diff _{S^\nu}(\lambda B)\geq \lambda \Diff _{S^\nu}( B).
\]
\end{lemma}
\begin{proof} The coefficients of $\Diff _{S^\nu}(\lambda B)$ are of the form 
\[
\frac {r-1+\lambda \sum n_ib_i}{r}\geq \lambda \left( \frac {r-1+\sum n_ib_i}{r}\right).\qedhere
\]
\end{proof}

\begin{theorem}[Kawamata Subadjunction] Let $(X,\Delta)$ be a pair such that $X$ is
quasi-projective and normal and $K_X+\Delta$ is $\mathbb{Q}$-Cartier.  Assume $V$ is a
subvariety such that $(X,\Delta)$ is log canonical at the generic point $\eta$ of $V$ and
$V$ is the only non-klt centre of $(X,\Delta)$ at $\eta$.  Then there is a
$\mathbb{Q}$-divisor $B\geq 0$ and a $\mathbb{Q}$-divisor class $J$ on the normalisation
$V^n$ of $V$, such that
\[
(K_X+\Delta )|_{V^n}\sim_{\mathbb{Q}}K_{V^n}+B+J.
\]
If $X$ is projective, then $J$ is pseudo-effective. 

Furthermore, if there is a generically finite morphism $\pi\colon\map Y.X.$ such that if
we write $f^*(K_X+\Delta)=K_{Y}+\Delta_Y$, $\Delta_Y$ is effective and
$\map W.\pi^{-1}(V).$ is a finite morphism on to a non-klt centre of $(X,\Delta)$ and we
denote by $p\colon\map W^n.V^n.$ the natural map between the normalisations, then if we
apply Kawamata subadjunction to $W^n$ and $(Y,\Delta_Y)$ and write
$(K_{Y}+\Delta_Y)|_{W^n}\sim_{\mathbb{Q}}K_{W^n}+B_{W}+J_W$, then we have $J_W=p^*J$.
\end{theorem}
\begin{proof} These statements follow from \cite[8.4-8.6]{Kollar07}, especially [8.4.9]
for the properties of $J$.  In particular, the last statement is an immediate consequence
of the results there: if we choose sufficiently high models $\alpha\colon\map V'.V^n.$ and
$\beta\colon\map W'.W^n.$, we also assume there is a (generically finite) morphism
$q\colon\map W'.V'.$, we know that there is a divisor class $J'$ on $V'$ such that we have
$J=\alpha_*J'$ and $J_W=\beta_*q^*J'$. But since $p\colon\map W^n.V^n.$ is a finite
morphism, this immediately implies that $p^*J=J_W$.
\end{proof}

\begin{lemma}\label{l-cover} Fix $q\in \mathbb{N}$ and let $I_0=\{\frac1q,..., \frac{q-1}{q}\}$.
Let $(X,\Delta)$ be $\mathbb{Q}$-factorial pair such that
$\coeff (\Delta )\subset I=D(I_0)$.  Then, for any point $x\in X$, there is a finite
morphism $\pi\colon\map Y.U.$ for some neighbourhood $x\in U\subset X$ with Galois group
$G$, such that $\pi^*(K_U+\Delta|_U)=K_Y+\Delta_Y$, and $Y$ is Gorenstein canonical, the
coefficients of $\Delta_Y$ are in $I_0$ and the components of $\Delta_Y$ are Cartier.
\end{lemma}
\begin{proof} We begin by constructing a finite cover $p_1\colon\map Y_1.U.$ of normal
varieties, which is \'etale in codimension $1$, such that $p_1^*K_X$ and $p_1^*\Delta_i$
are Cartier for all components $\Delta_i$ of $\Delta$.  To this end, take $D$ to be either
$K_X$ or a component of $\Delta$.  Then $D$ is a Weil divisor.  Let $n$ be the smallest
integer such that $nD$ is Cartier in a neighbourhood $U$ of $x\in X$.  Pick an isomorphism
$\ring U.(nD|_U)\cong \ring U.$ and let $\pi\colon\map U'.U.$ be the normalisation of the
corresponding cyclic cover.  Then $D'=\pi ^{-1}(D|_U)$ is Cartier.  Since $D$ is Cartier
in codimension $1$ (as $X$ is normal and hence $R_1$), it follows that $\map U'.U.$ is
\'etale in codimension $1$.  We let $p_1\colon\map Y_1.U.$ be the normalisation of the
fibre product of these cyclic covers.  Note that as $p_1$ is \'etale in codimension $1$,
writing $K_{Y_1}+\Delta _1=(K_X+\Delta )|_{Y_1}$, we have $\coeff(\Delta_1)\subset I$.

If $\Delta^i$ is a component of the support of $\Delta _1$ then the coefficient of
$\Delta ^i$ in $\Delta _1$ is of the form $\frac{m_i-1+a_i}{m_i}$ where $a_i= {r_i}/ q$
from some positive integers $m_i, r_i$.  By Kawamata's trick, we may take a branched cover
$\mu\colon\map Y.Y_1.$, which is branched of degree $m_i$ along each $\Delta_i$.  Let
$P_1$ be the generic point of $\Delta ^i$ and $P$ be the generic point of an irreducible
component of $\mu^{-1}(P)$.  If $K_Y+\Delta _Y=\mu ^*(K_{Y_1}+\Delta_1 )$ then by an easy
local computation one sees that $\mult_P(\Delta_Y )=r_i/q$ as required.
\end{proof}

\begin{theorem}\label{t-4.2} Fix $n$ and $q\in \mathbb{N}$.  Let $I=D(I_0)$ where
$I_0=\{\frac1q,..., \frac{q-1}{q}\}$ for some $q\in \mathbb{N}$.  Let $(X,B)$ be an
$n$-dimensional projective Kawamata log terminal log pair with $\coeff (B)\subset I$.

Assume that there is a flat projective family $h\colon\map \mathcal{V}.S.$ over a smooth
base $S$ with a generically finite morphism $\beta\colon\map \mathcal{V}.X.$ such that for
a general point $v\in \mathcal{V}$, there exists a $\mathbb{Q}$-Cartier divisor
$B'\geq 0$, such that $(X,B+B')$ is log canonical in a neighbourhood of $\beta(v)$ and if
$s=h(v)$ then the fibre $\mathcal{V}_s=\mathcal{V}\fb S. \{s\}$ is mapped isomorphically
to the unique non-klt centre $V$ of $(X,B+B')$ containing $\beta(v)$.

If $\nu\colon\map W.V.$ is the normalisation then there is a divisor $\Theta $ on $W$ such
that $\coeff (\Theta )\in I$ and
\[
(K_X+B+B')|_W-(K_W+\Theta )
\]
is pseudo-effective.  Moreover, there is a log resolution $\psi\colon\map W'.W.$ of
$(W,\Theta )$ such that 
\[
K_{W'}+\Omega \geq (K_X+B)|_{W'}
\] 
where $\Omega =\psi ^{-1}_*\Theta +\Exc(\psi )$.
\end{theorem}
\begin{proof} Let $\mathcal{W}$ be the normalisation of $\mathcal{V}$.  Replacing $S$ by a
dense open subset we may assume that for any $s\in S$, the fibre
$\mathcal{W}_s:=\mathcal{W}\fb S.\{s\}$ is the normalisation of $\mathcal{V}_s=V$.  We
denote $\mathcal{W}_s$ by $W$.

For any point $P$ on $W$, if we consider $\nu(P)\in V$ as a point in $X$ then there exists
a Zariski open subset $U_i \subset X$ containing $\nu(P)$ with a finite morphism
$\pi_i\colon\map Y_i.U_i.$ as in Lemma \ref{l-cover}.  Let $G=G_i$ be the corresponding
Galois group and we denote by
\[
\pi_i^*(K_{U_i}+B|_{U_i})=K_{Y_i}+\Delta_{Y_i}.
\]
Consider $\mathcal{V}_i:=\mathcal{V}\times_X Y_i$ and its normalisation
$\map\mathcal{W}_i.\mathcal{V}_i.$.  Let
\[
K_{\mathcal{W}_i}+\Psi'_i=(K_X+B)|_{\mathcal{W}_i}=(K_{Y_i}+\Delta_{Y_i})|_{\mathcal{W}_i}\qquad \text{and}\qquad \Psi_i=\Psi'_i\vee 0.
\]
Let $F=(\mathcal{W}_i)_s$ and $E=(\mathcal{V}_i)_s$ be the fibres of $\mathcal{W}_i$ and
$\mathcal{V}_i$ over a general point $s\in S$.  Then $\map E.\pi_i^{-1}(V\cap U_i).$ is an
isomorphism and $\pi_i^{-1}(V\cap U_i)$ is a union of non-klt centres of
$(Y_i, \Delta_{Y_i}+\pi_i^*B'|_{U_i})$ with an induced $G$-action and $F$ is the
normalisation of $E$.

We denote $\Psi_i|_F=\Phi_i$.  We note that since $F$ is a general fibre of
${\mathcal{W}_i}$ over $S$, we have $K_{\mathcal{W}_i}|_{F}=K_F$. There is a natural
isomorphism
\[
W_{U_i}:=W\times_X U_i\cong F/G,
\] 
so we have a morphism $p_i\colon\map F.W.$.  We define $\Theta_i$ on $W_{U_i}$ via the
equality
\[
p_i^*(K_{W_{U_i}}+\Theta_i)=K_{F}+\Phi_i.
\]
Note in fact that $K_F+\Psi_i'|_F=(K_X+B)|_F$ is pulled back from $W_{U_i}$, each
component of the support of $\Psi -\Psi'$ is Cartier (by Lemma \ref{l-cover}),
$(\Psi -\Psi')|_F$ is $G$-invariant and so it is a pull back of a $\mathbb{Q}$-Cartier
divisor on $F/G$.  Since $\Phi_i\geq 0$, it follows that $\Theta_i\geq 0$.

We choose finitely many open subsets $\{U_i\}_{i\in I}$ covering $X$, and we define the
divisor $\Theta$ on $W$ by defining the coefficient of any prime divisor $P$ in $\Theta$
to be
\[
\mult_P(\Theta )=\max\{\, \mult_P(\Theta _i) \,|\, P\in W_{U_i} \,\}
\]
We now check that $(W,\Theta)$ is the pair we are looking for.  

We first check that the coefficients of $\Theta$ are in $I=D(I_0)$.  It suffices to show
that the coefficients of $\Psi_i$ are in $I_0$, as this implies that the coefficients of
$\Phi_i$ are in $I_0$, and we can then conclude from the usual Hurwitz formula that the
coefficients of $\Theta$ are in $I=D(I_0)$.  By our construction $\map \mathcal{W}_i.Y_i.$
is generically finite, and
\[
K_{\mathcal{W}_i}+\Psi_i'=(K_{Y_i}+\Delta_{Y_i})|_{\mathcal{W}_i}.
\]  
Since $Y_i$ is Gorenstein, the components of $\Delta_{Y_i}$ are Cartier and all their
coefficients are in $I_0$, it follows that all coefficients of $\Psi'_i$ are of the form
$m+\sum m_ji_j$ where $m$, $m_j\in \mathbb{N}$ and $i_j\in I_0$ and so they belong to
$\frac 1 q \mathbb{N}$.  Since $K_X+B$ is Kawamata log terminal,
$K_{\mathcal{W}_i}+\Psi_i'$ is sub Kawamata log terminal and so
$\coeff(\Psi _i)\subset \frac 1 q \mathbb{N}\cap [0,1)=I_0$.

Next we check that 
\[
(K_X+B+B')|_W-(K_W+\Theta )
\]
is pseudo-effective.  By Kawamata subadjunction, we may write
\[
(K_X+B+B')|_W=K_W+(B+B')_W+J_W
\] 
where $(B+B')_W\geq 0$ is an $\mathbb{R}$-divisor and $J_W$ is a pseudo-effective
$\mathbb{R}$-divisor defined up to $\mathbb{R}$-linear equivalence.  Define $(B')_W$ by
the equation $(B+B')_W=B|_W+(B')_W\geq 0$.  Since $s$ is general, we may assume that $E$
does not belong to $\Supp(\Delta_{Y_i})$.  Then $F$ is a union of log canonical centres of
$K_{\mathcal W _i}+B'|_{\mathcal W _i}$.  Applying Kawamata subadjunction, we may write
\[
\big(K_{Y_i}+\pi_i^{*}(B'|_{U_i})\big)|_{F}=K_{F}+B'_F+J_F,
\]
where $B'_F\ge 0$ and since $\map F.W_{U_i}.$ is finite, we have $J_F=p_i^*(J_W)$.

If we let $K_{\mathcal{W}_i}+\Gamma_i=(K_{Y_i})|_{\mathcal{W}_i}$ then we know that the
discrepancy of any divisor with respect to $(\mathcal{W}_i,\Gamma_i)$ is positive as
$K_{Y_i}$ is canonical. In particular, $\Gamma_i\le 0$. Thus
\[ 
\Phi_i = ((K_{Y_i}+\Delta _{Y_i})|_{\mathcal{W}_i}-K_{\mathcal{W}_i})|_F= (\Gamma_i+\Delta_{Y_i}|_{\mathcal{W}_i})|_{F}\vee 0\le   \Delta_{Y_i}|_{F}.
\]

By Kawamata subadjunction, we know
that 
\[
(K_X+B+B')|_W-(K_{W}+\Theta+J_W)=(B+B')_W-\Theta
\] 
is a well defined $\mathbb{Q}$-divisor on $W$, which we claim to be effective. For this
purpose, we only need to check the multiplicities at each codimension 1 point $P$ on
$W$. It suffices to verify this after pulling back via $\map Y_i.X.$, where $U_i$ contains
$P$ and $\mult_P(\Theta )=\mult_P(\Theta _i)$.

We have
\begin{align*} 
&p_i^*\big((K_X+B+B')|_W-(K_{W}+\Theta+J_W)\big)\\
=&\big(K_{Y_i}+\Delta_{Y_i}+\pi_i^*(B'|_{U_i})\big)|_F-(K_F+\Phi_i+J_F)\\
=& B'_F+\left(\Delta_{Y_i}|_F-\Phi_i\right)\ge 0.
\end{align*} 
where we used the fact that $p^*_iJ_W=J_F$ in the first equality.

Finally we check the last statement.  Let 
\[
W'_{U_i}:= W_{U_i}\fb W. W' \quad \psi _i\colon\map W'_{U_i}.W_{U_i}.\quad \text{and} \quad \Omega_i =\psi ^{-1}_{i,*}\Theta_i +\Exc(\psi _i ).
\] 
Note that $\Omega|_{W'_{U_i}}\ge \Omega_i$.  After possibly shrinking $S$, we can assume
that there is a resolution $\map \mathcal W'.\mathcal W.$ which induces a log resolution
of each fibre and in particular, $\map \mathcal W_s'.\mathcal W_s.$ induces the log
resolution $\map W'.W.$ of $(W,\Theta )$.  We may also assume that there is a
$G$-equivariant resolution $\map \mathcal W '_i.\mathcal W_i.$ over $S$ which gives a
$G$-invariant log resolution $\psi_F\colon\map F'.F.$, such that there is a proper
morphism $\map F'/G.W'_{U_i}.$ .  Let
\[
\Omega_{F'}=\psi_{F*}^{-1}(\Theta_i)+\Exc({\psi_F}).
\]
It follows that
\[
K_{F'}+\Omega_{F'}-(K_{\mathcal{W}_i}+\Psi'_i)|_{F'}\geq 0,
\]
as $(K_{\mathcal{W}_i}+\Psi')|_F$ is Kawamata log terminal.  Let $\Omega_{F'/G}$ be the
$\mathbb{Q}$-divisor defined by $(K_{F'/G}+\Omega_{F'/G})|_{F'}=K_{F'}+\Omega_{F'}$.  Then
\[
K_{F'/G}+\Omega_{F'/G}-(K_X+B)|_{F'/G}\geq 0.
\]
The claim follows by pushing forward to $W'_{U_i}$. 
\end{proof}

\subsection{DCC of volumes and birational boundedness}

In this section we prove a result on the ACC for volumes and on log birational boundedness
of pairs $(X,B)$ with hyperstandard coefficients.  The general case, Theorem \ref{t-bb} is
covered in the next section.

\begin{theorem}\label{t-bbw} Fix $n\in\mathbb{N}$ and a finite set $I_0\subset [0,1]\cap \mathbb{Q}$.  
Let $J=D(I_0)\subset [0,1]$ and $\mathfrak{D}$ be the set of projective log canonical
pairs $(X,B)$ such that $\dim X=n$ and $\coeff(B)\subset J$.

Then there is a constant $\delta >0$ and a positive integer $m$ such that
\begin{enumerate}
\item the set 
\[
\{\, \vol(X,K_X+B) \,|\, (X,B)\in \mathfrak{D} \,\}
\]
also satisfies the DCC,
\item if $\vol(X,K_X+B)>0$ then $\vol(X,K_X+B)\geq \delta$, and
\item if $K_X+B$ is big then $\phi _{m(K_X+B)}$ is birational.
\end{enumerate}
\end{theorem}
\begin{proof} We may assume that $I_0=\{\, \frac jq \,|\, 1\leq j\leq q \,\}$ for some
$q\in \mathbb{N}$.  We proceed by induction on the dimension.  Replacing $X$ by a log
resolution and $B$ by its strict transform plus the exceptional divisor, we may assume
that $(X,B)$ is log smooth.  Replacing $B$ by $\{B\}+(1-\frac 1 r)\lfloor B\rfloor$ for
some $r\gg 0$, we may assume that $(X,B)$ is Kawamata log terminal.  Replacing $(X,B)$ by
the log canonical model \cite{BCHM10}, we may assume that $K_X+B$ is ample.

By induction, there is a positive integer $l\in \mathbb{N}$ such that if $(U,\Psi)$ is a
projective log canonical pair of dimension $\leq n-1$, $\coeff (\Psi )\subset J$, and
$K_U+\Psi$ is big then $\phi _{l(K_U+\Psi )}$ is birational.  Fix $k\in \mathbb{N}$ such
that
\[ 
\vol(X,k(K_X+B))> (2n)^n.
\]

\begin{claim}\label{c-temp} There is an integer $m_0>0$ such that $\phi_{m_0k(K_X+B)}$ is
birational.
\end{claim}
\begin{proof} By Lemma \ref{l-lcc}, there is a family $\map V.T.$ of subvarieties of $X$
such that for any two general points $x$, $y\in X$ there exists $t\in T$ and
$0\leq D_t\sim _\mathbb{R} k(K_X+B)$ such that $(X,B+D_t)$ is not Kawamata log terminal at
both $x$ and $y$ and there is a unique non Kawamata log terminal place whose centre $V_t$
contains $x$ (so that in particular $(X,B+D_t)$ is log canonical at $x$).  Let
$\nu \colon\map V^\nu.V.=V_t$ be the normalisation.  By Theorem \ref{t-4.2} there exists a
$\mathbb{Q}$-divisor $\Theta$ on $V^\nu$ such that
\begin{enumerate}
\item $(K_X+B+D_t)|_{V^\nu}-(K_{V^\nu}+\Theta)$ is pseudo-effective,
\item $\coeff (\Theta )\subset D(I_0)$, and
\item there exists $\psi \colon\map U.V^\nu.$ a log resolution of $(V^\nu,\Theta )$ such
that if $\Psi =\psi ^{-1}_*\Theta +\Exc(\psi )$ then $(K_U+\Psi )\geq (K_X+B)|_U$.  In
particular $(K_U+\Psi)$ is big.
\end{enumerate} 

By induction, $\phi _{l(K_U+\Psi)}$ is birational and so $\phi _{l(K_{V^\nu}+\Theta )}$ is  birational
as well. 

Let $x$, $y\in X$ be general points.  We may assume that $x\in V$ is a general point.  If
$v=\dim V$ and $H_i\in |l(K_{V^\nu}+\Theta )|$ are general divisors passing through $x$
and we set
\[
H=\frac v{v+1} (H_1+\ldots+ H_{v+1}),
\] 
then $x$ is an isolated component of the non Kawamata log terminal locus of
$(V^\nu,\Theta +H)$.  If in addition $y\in V$, then, arguing as above, and at the expense
of replacing $H$ by something $\mathbb{Q}$-linearly equivalent to $3H$ (or indeed any
multiple of $H$ greater than $2$), we may arrange that $(V^\nu,\Theta +H)$ is not log
canonical at $y$.  Since
\[
(K_X+B+D_t)|_{V^\nu}-(K_{V^\nu}+\Theta )
\]
is pseudo-effective, we may pick
\[
\tilde {H}\sim _\mathbb{R} {(k+1)vl}(K_X+B)
\] 
such that $\tilde H|_V=H$ in a neighbourhood of $x\in V$. Let
$\lambda=\lct _x(V^\nu,\Theta;H)$, then $\lambda \leq 1$.   

By inversion of adjunction
\begin{enumerate} 
\item $x\in X$ is a non Kawamata log terminal centre of $(X,B+D_t+\lambda \tilde H)$,
\item $(X,B+D_t+\lambda \tilde H)$ is log canonical at $x\in X$, and
\item $(X,B+D_t+\lambda \tilde H)$ is not Kawamata log terminal at $y\in X$ and 
it is not log canonical at $y$ if $y\in V$.  
\end{enumerate}

Since $K_X+B$ is big, we may write $K_X+B\sim _\mathbb{R} A+E$ where $E\geq 0$ and $A$ is
ample.  As $x$ and $y$ are general we may assume that $x$ and $y$ don't belong to $E$.  If
$y$ does not belong to $V$ then $x$ and $y$ belong to different connected components of
the non kawamata log terminal locus.  In this case we can increase $\lambda$ a little bit
and use $A$ to make $(X,B+D_t+\lambda \tilde H)$ not log canonical at $y$.  Using $A$ to
tie-break (cf. Proposition \ref{p-mincentre}) we may assume that $x$ is an isolated non
kawamata log terminal centre.

By Lemma \ref{l-pb}, $\phi _{K_X+\lceil t(K_X+B)\rceil }$ is birational for any
$t\geq (k+1)v$.  We claim there is an inequality
\[
(m+1)(K_X+B)\geq K_X+\lceil m(K_X+B)\rceil
\] 
for any integer $m>0$ which is divisible by $q$.  Grant this for the time being.  It
follows that that $\phi _{(m+1)(K_X+B)}$ is birational for any integer $m$ divisible by
$q$ such that $m> (k+1)v$.  Claim \ref{c-temp} now follows.

To see the inequality note that if $k\in \mathbb{N}$, then
$\lceil k/r\rceil \leq k/r+(r-1)/r$ and so since $m/q\in \mathbb{N}$, we have
\begin{align*} 
\lceil m(\frac{r-1+\frac{a}{q}}{r})\rceil &\leq m(\frac{r-1+\frac{a}{q}}{r})+\frac {r-1}r\\
                                          &\leq  (m+1)(\frac{r-1+\frac{a}{q}}{r}).
\end{align*} 
Since the coefficients of $B$ are of the form $\frac{r-1+\frac aq}r$, we have shown that
$(m+1)B\geq \lceil mB\rceil$.
\end{proof}

If $\vol(X,K_X+B)\geq 1$ then let $k=2(n+1)$.  The result follows in this case.
Therefore, we may assume that $\vol(X,K_X+B)<1$ and we pick $k\in \mathbb{N}$ such that
\[
(2n)^n\leq \vol(X,k(K_X+B))< (4n)^n.
\]
But then $\vol (X,m_0k(K_X+B))\leq (4m_0n)^n$.  By Theorem \ref{t-volddc1} the set of
these volumes satisfies the DCC and so there exists a constant $\delta >0$ such that
$\vol (K_X+B)\geq \delta$.  In particular
\[
k=\max (\lfloor \frac {2n}\delta\rfloor+1,2(n+1)).\qedhere
\]
\end{proof}

\section{Birational boundedness: the general case}

The purpose of this section is to prove \cite[1.4]{HMX14}:

\begin{theorem}\label{t-bb} Fix $n\in\mathbb{N}$ and a set $I\subset [0,1]$ which satisfies the DCC. 
Let $\mathfrak{D}$ be the set of projective log canonical pairs $(X,B)$ such that
$\dim X=n$ and $\coeff (B)\subset I$.

Then there is a constant $\delta >0$ and a positive integer $m$ such that
\begin{enumerate}
\item the set 
\[
\{\, \vol(X,K_X+B)\ \,|\, (X,B)\in \mathfrak{D} \,\}
\]
also satisfies the DCC,
\item if $\vol(X,K_X+B)>0$ then $\vol(X,K_X+B)\geq \delta$, and
\item if $K_X+B$ is big then $\phi _{m(K_X+B)}$ is birational.
\end{enumerate}
\end{theorem}

The proof is by induction on the dimension.  We will prove the following four statements (\cite{HMX14}).
\begin{theorem}[Boundedness of the anticanonical volume]\label{t-b} Fix $n\in \mathbb{N}$ and a 
set $I\subset [0,1)$ which satisfies the DCC.  Let $\mathfrak{D}$ be the set of Kawamata
log terminal pairs $(X,B)$ such that $X$ is projective, $\dim X=n$, $K_X+B\equiv 0$, and
$\coeff (B)\subset I$.

Then there exists a constant $M>0$ depending only on $n$ and $I$ such that
$\vol(X,-K_X)<M$ for any pair $ (X,B)\in \mathfrak{D}$.
\end{theorem}

\begin{theorem}[Effectively birational]\label{t-c} Fix $n\in \mathbb{N}$ and a set
$I\subset [0,1]$ which satisfies the DCC.  Let $\mathfrak{B}$ be the set of log canonical
pairs $(X,B)$ such that $X$ is projective, $\dim X=n$, $K_X+B$ is big, and
$\coeff (B)\subset I$.

Then there exists a positive integer $m=m(n,I)$ such that $\phi _{m(K_X+B)}$ is birational
for any $(X,B)\in \mathfrak{B}$.
\end{theorem}

\begin{theorem}[The ACC for numerically trivial pairs]\label{t-d} Fix $n\in \mathbb{N}$
and a DCC set $I\subset [0,1]$.

Then there is a finite subset $I_0\subset I$ such that if 
\begin{enumerate}
\item $(X,B)$ is an $n$-dimensional projective log canonical pair,
\item $\coeff (B)\subset I$, and 
\item $K_X+B\equiv 0$, 
\end{enumerate}then the coefficients of $B$ belong to $I_0$.
\end{theorem}

\begin{theorem}[The ACC for the LCT]\label{t-a} Fix $n\in \mathbb{N}$ and a  set
$I\subset [0,1]$ which satisfies the DCC.  Then there exists a constant $\delta >0$ such
that if
\begin{enumerate}
\item $(X,B)$ is an $n$-dimensional log pair with $\coeff (B)\in I$,
\item $(X,\Phi )$ is Kawamata log terminal for some $\Phi\ge 0$ and  
\item $B'\geq (1-\delta )B$ where $(X,B')$ is a log canonical pair,
\end{enumerate}
then $(X,B)$ is log canonical.
\end{theorem}

\begin{proof}[Proof of Theorems \ref{t-b}, \ref{t-c}, \ref{t-d}, and \ref{t-a}] The proof
is by induction on the dimension.  The case $n=1$ is obvious.  The proof subdivided into
the following 4 steps.
\begin{enumerate} 
\item Theorems \ref{t-d} and \ref{t-a} in dimension $n-1$ imply Theorem \ref{t-b} in
dimension $n$ (cf. Theorem \ref{t-da-b}),
\item Theorem \ref{t-c} in dimension $n-1$ and Theorem \ref{t-b} in dimension $n$ imply
Theorem \ref{t-c} in dimension $n$ (cf. Theorem \ref{t-cb-c}),
\item Theorem \ref{t-d} in dimension $n-1$ and Theorem \ref{t-c} in dimension $n$ imply
Theorem \ref{t-d} in dimension $n$ (cf. Theorem \ref{t-dc-d}), and
\item Theorems \ref{t-b}, \ref{t-c}, \ref{t-d} and \ref{t-a} in dimension $n-1$ imply
Theorem \ref{t-a} in dimension $n$ (cf. Theorem \ref{t-c-a}). \qedhere
\end{enumerate}
\end{proof}

\begin{proof}[Proof of Theorem \ref{t-bb}] (3) follows from Theorem \ref{t-c}.  To prove
(1), we may fix $M>0$ and consider pairs $(X,B)$ such that $0<\vol (X,K_X+B)\leq M$.  By
Proposition \ref{p-BB} the pairs $(X,B)$ are log birationally bounded. (1) now follows
from Corollary \ref{c-voldcc} and (2) is an easy consequence of (1).
\end{proof}

\subsection{Boundedness of the anticanonical volume}

\begin{theorem} \label{t-da-b} Theorems \ref{t-d} and \ref{t-a} in dimension $n-1$ imply
Theorem \ref{t-b} in dimension $n$.
\end{theorem}
\begin{proof} Suppose that $(X,B)$ is a pair in $\mathfrak{D}$ with $\vol(X,-K_X)>0$.  By
Theorem \ref{t-model}, there exists a small proper birational morphism
$\nu \colon\map X'.X.$ such that $X'$ is $\mathbb{Q}$-factorial.  Let
\[
K_{X'}+B'=\nu ^*(K_X+B)\equiv 0.
\]  
Then $(X',B')\in \mathfrak{D}$ and $\vol(X',-K_{X'})=\vol(X,-K_X)$.  Replacing $X$ by
$X'$, we may therefore assume that $X$ is $\mathbb{Q}$-factorial.  If $x\in X$ is a
general point, then by a standard argument (cf. \cite[10.4.12]{Lazarsfeld04b}), there
exists $G\sim _{\mathbb{R}}-K_{X}$ with
\[
\mult_x(G) >\frac 1 2 \big(\vol(X,-K_X)\big)^{1/n}.
\]  
It follows that
\[
\sup \{\, t\geq 0 \,|\, \text{$(X,tG)$ is log canonical} \,\} < \frac {2n} {\big(\vol(X,-K_X) \big)^{1/n}}
\]
(cf. \cite[9.3.2]{Lazarsfeld04b}).  Therefore, we may assume that
\[
(X,\Phi:=(1-\delta )B+\delta G)
\] 
is log canonical but not Kawamata log terminal for some 
\[
\delta <2n/(\vol(X,-K_X))^{1/n}.
\]
Note that 
\[
K_X+\Phi\sim_\mathbb{R} (1-\delta)(K_X+B)\equiv 0.
\]  
By tie breaking (cf. Proposition \ref{p-mincentre}), we may assume that $(X,\Phi)$ has a
unique non Kawamata log terminal centre say $Z$ with a unique non Kawamata log terminal
place say $E$.  Let $\nu\colon\map X'.X.$ be the corresponding divisorial extraction so
that $\rho(X'/X)=1$ (cf. Theorem \ref{t-model}(3)) and the exceptional locus is given by
the prime divisor $E\subset X'$.  Let $\Phi'=\nu^{-1}_*\Phi$ and $B'=\nu ^{-1}_*B$ and
write
\[
K_{X'}+B'+aE=\nu ^*(K_X+B),\qquad K_{X'}+\Phi '+E=\nu ^*(K_X+\Phi)
\]
where $a<1$.  In particular $ K_{X'}+\Phi '+E$ is purely log terminal.  We now run the
$K_{X'}+\Phi '\equiv -E$ minimal model program $\psi\colon\rmap X'.X''.$ (cf. Proposition
\ref{p-MF}) until we obtain a Mori fibre space $\pi \colon\map X''.W.$.  Let
$E''=\psi _*E$, $\Phi''=\psi _*\Phi$, and $B''=\psi _*B'$, so that $E''$ is $\pi$-ample.
Note that since $K_{X'}+\Phi '+E\equiv 0$, it follows that $K_{X''}+\Phi '' +E''$ is
purely log terminal.  After restricting to a general fibre, we may assume that $E''$ is
ample and we write
\[
(K_{X''}+B''+E'')|_{E''}=K_{E''}+B_{E''},\qquad (K_{X''}+\Phi ''+E'')|_{E''}=K_{E''}+\Phi _{E''}.
\]
Note that 
\begin{enumerate}
\item $\coeff (B_{E''})\subset D (I)$ (by Theorem \ref{tSLA}, since
$\coeff (B'')\subset I$),
\item $K_{E''}+\Phi _{E''}$ is Kawamata log terminal (since $K_{X''}+\Phi '' +E''$ is
purely log terminal), and
\item $\Phi _{E''}\geq (1-\delta )B_{E''}$ (by Lemma \ref{l-diff}, since
$\Phi ''\geq (1-\delta )B''$).
\end{enumerate} 

If $\vol(X,-K_X)\gg 0$, then $\delta \ll 1$, and so by Theorem \ref{t-a} in dimension
$\leq n-1$, we have that $K_{E''}+B_{E''}$ is log canonical.

Since $\Phi ''\geq (1-\delta )B''$ and $K_{X''}+\Phi''+ {E''}\equiv 0$, we have 
\[
K_{X''}+(1-\eta)B''+{E''}\equiv _W 0 \qquad \text{for some} \qquad 0<\eta<\delta,
\] 
and so 
\[
K_{E''}+\Diff _{E''} \big((1-\eta)B'' \big)\equiv 0.
\]  
By  Lemma \ref{l-diff}, 
\[
\Diff _{E''} \big((1-\eta)B'' \big)\geq (1-\eta)B_{E''}.
\]
We claim that $0$ is not an accumulation point for the possible values of $\eta$.  If this
where not the case then there would be a decreasing sequence $\eta _k>0$ with
$\lim \eta _k=0$.  But then, it is easy to see that the coefficients of
$\Diff_{E''}((1-\eta_k )B'')$ belong to a DCC set and so we obtain a contradiction by
Theorem \ref{t-d} in dimension $n-1$.  Since 
\[
\eta<\delta \leq 2n/\vol(X,-K_X),
\] 
it follows that $\vol(X,-K_X)$ is bounded from above.
\end{proof}

\subsection{Birational boundedness}

\begin{theorem}\label{t-psef} Assume that Theorem \ref{t-c} holds in dimension $n-1$ and
Theorem \ref{t-b} holds in dimension $n$.  Then there is a constant $\beta <1$ such that
if $(X,B)$ is an $n$-dimensional projective log canonical pair where $K_X+B$ is big and
$\coeff (B)\subset I$, then the pseudo-effective threshold satisfies
\[
\lambda :=\inf \{\, t\in \mathbb{R} \,|\, \text{$K_X+tB$ is big} \,\}\leq \beta.
\]
\end{theorem}
\begin{proof} Suppose that we have a sequence of pairs $(X_i,B_i)$ with increasing
pseudo-effective thresholds $\lambda _i< \lambda _{i+1}$ such that $\lim \lambda _i=1$.
In particular we may assume that $1>\lambda _i\geq 1/2$.

\begin{claim} We may assume that there is a sequence of $\mathbb{Q}$-factorial projective
Kawamata log terminal pairs $(Y_i,\Gamma _i)$ such that $\coeff (\Gamma _i)\subset I$,
$-K_{Y_i}$ is ample, $K_{Y_i}+\lambda _i\Gamma _i\equiv 0$ and $\dim Y_i\leq n$.
\end{claim}
\begin{proof} We may assume that $1\in I$.  As a first step, we will show that we may
assume that $(X,B)=(X_i,B_i)$ is log smooth.  Let $\nu \colon\map X'.X.$ be a log
resolution of $(X,B)$ and write 
\[
K_{X'}+B'=\nu ^*(K_X+B)+E
\] 
where $B'=\nu ^{-1}_*B+\Exc(\nu)$.  Note that $K_{X'}+B'$ is big and if $K_{X'}+tB'$ is
big, then so is $K_X+tB=\nu _*(K_{X'}+tB')$. Thus
\[
1>\lambda ':=\inf \{\, t\in \mathbb{R} \,|\, \text{$K_{X'}+tB'$ is big} \,\}\geq \lambda
\]
and we may replace $(X,B)$ by $(X',B')$.  Therefore we may assume that $(X,B)$ is log
smooth.

Since $K_X+B$ is big, we may pick an effective $\mathbb{Q}$-divisor
$D\sim _\mathbb{Q} K_X+B$ and so for $0<\epsilon \ll 1$ we have
\[
(1+\epsilon)(K_X+\lambda B)\sim _\mathbb{Q} K_X+\mu B+\epsilon D
\] 
where $0<\mu: =\lambda (1+\epsilon)-\epsilon <\lambda$ and $K_X+\mu B+\epsilon D$ is
Kawamata log terminal.  Since $\mu B+\epsilon D$ is big, by \cite{BCHM10} we may run the
$K_X+\mu B+\epsilon D$ minimal model program say $f\colon\rmap X.X'.$.  Since this is also
a $K_X+\lambda B$ minimal model program, we may assume that $K_{X'}+\lambda B'$ is nef and
Kawamata log terminal where $B'=f_*B$ and $D'=f_*D$.
 
We may now run a $K_{X'}+\mu B'$ minimal model program.  By Proposition \ref{p-MF} after
finitely many $K_{X'}+\mu B'+\epsilon D'$ flops $g\colon\rmap X'.X''.$ we obtain a
$K_{X''}+\mu B''+\epsilon D ''$-trivial contraction of fibre type $\map X''.Z.$ (where
$B''=g_*B'$ and $D''=g_*D'$) such that $\epsilon D'' $ is ample over $Z$.  Therefore
$-K_{X''}$ is ample over $Z$ (since $ B''\geq 0$ and $\rho (X''/Z)=1$).  It follows that
$K_{X''}+\lambda B''$ is Kawamata log terminal and $K_{X''}+\lambda B''\equiv _Z 0$.
Letting $(Y,\Gamma )=(F, B''|_F)$ where $F$ is a general fibre of $\map X''.Z.$, the claim
follows.
\end{proof}

Let $\nu _i\colon\map Y'_i.Y_i.$ be a log resolution of $(Y_i,\Gamma _i)$,
$D_i=(\Gamma _i)_{\rm red}$ and $\Gamma '_i$ (resp. $D'_i$) be the strict transform of
$\Gamma _i$ (resp. $D_i$) plus the $\nu _i$-exceptional divisors.  Since
$(Y_i,\lambda _i\Gamma _i)$ is klt, then for any $0<\delta \ll 1$,
\[
K_{Y'_i}+\Gamma '_i\geq  \nu _i^*(K_{Y_i}+(\lambda _i+\delta)\Gamma _i)\equiv \delta\nu _i^* \Gamma _i
\]
and so both $K_{Y'_i}+\Gamma '_i$ and $K_{Y'_i}+D'_i$ are big.  By Theorem \ref{t-b} in
dimension $n$ there exists a constant $C$ such that $\vol(Y_i,\lambda _i \Gamma_i)<C$.
Since $I$ satisfies the DCC, there exists a smallest non-zero element $\alpha \in I$.

\begin{claim} The pairs $(Y'_i,D'_i)$ are log birationally bounded.
\end{claim}
\begin{proof} Since $K_{Y'_i}+D'_i$ is big, then so is $K_{Y'_i}+\frac {r-1}r D'_i$ for
any $r\gg 0$.  If $\Theta _i:=\frac {r-1}r D'_i$ then
\begin{align*} 
\vol(Y'_i, K_{Y'_i}+\Theta _i)&\leq \vol(Y'_i,K_{Y'_i}+D'_i)\\
                             &\leq \vol(Y_i,K_{Y_i}+D_i)\\
                             &=\vol(Y_i,D_i-\lambda _i \Gamma _i)\\ 
                             &\leq \vol(Y_i,D_i)\\
                             &\leq \vol(Y_i,\frac {1} \alpha \Gamma _i)\\
                             &\leq \frac C {(\lambda _i \alpha )^{n'}}\\
                             &\leq C\left( \frac 2 \alpha \right)^{n'} 
\end{align*} 
where $n'=\dim Y_i$.  Since 
\[
\coeff(\Theta _i)\subset \{\, 1-\frac 1 r \,|\, r\in \mathbb{N}\,\}
\]
then by Theorem \ref{t-bbw}, it follows that there exists a constant $m>0$ such that
$\phi_{m(K_{Y'_i}+\Theta _i)}$ is birational.  By Proposition \ref{p-BB}, the pairs
$(Y'_i,\Theta _i)$ are log birationally bounded and hence so are the pairs $(Y'_i, D'_i)$.
\end{proof} 

By Corollary \ref{c-voldcc}, it follows that there exists a constant $\delta >0$ such
that 
\[
\vol(Y'_i,K_{Y'_i}+\Gamma'_i)\geq \delta.
\]
But then we have
\begin{align*} 
\delta &\leq\vol(Y'_i,K_{Y'_i}+\Gamma '_i)\\
       &\leq\vol(Y_i,K_{Y_i}+\Gamma _i)\\
       &=\vol(Y_i,\frac{1-\lambda _i}{\lambda _i}\lambda_i\Gamma _i)\\
       &=(\frac{1-\lambda _i}{\lambda _i})^{n'}\vol(Y_i,\lambda _i \Gamma _i)\\
       &\leq (\frac{1-\lambda _i}{\lambda _i})^{n'}C.
\end{align*} 
Thus, if
\[
\beta=\frac 1 {1+(\frac \delta C )^{1/n'}}<1,
\] 
then $\lambda_i\leq \beta$.
\end{proof}

\begin{theorem}\label{t-cb-c} Theorem \ref{t-c} in dimension $n-1$ and Theorem \ref{t-b}
in dimension $n$ imply Theorem \ref{t-c} in dimension $n$.
\end{theorem}
\begin{proof} By Theorem \ref{t-psef} there exists a constant $\gamma <1$, such that
$K_X+\gamma \Delta$ is big.  Fix a positive integer $q$, such that
$(1-\gamma)\delta>\frac 1 q$, where $\delta=\min(I\cap (0,1])$ and let
$I_0=\{ \frac{1}{q},\frac 2 q,\ldots,\frac {q-1}q,1\}$.  It is easy to see that there
exists a $\mathbb{Q}$-divisor $\Delta _0$ such that
\[
\gamma \Delta\leq \Delta_0\leq \Delta \qquad \text{and} \qquad \coeff ( \Delta_0)\subset I_0.
\]
By Theorem \ref{t-bbw}, there exists a constant $m\in \mathbb{N}$ such that
$\phi_{m(K_X+\Delta_0)}$ is birational.  Since $\Delta_0\leq \Delta$,
$\phi _{m(K_X+\Delta)}$ is also birational.
\end{proof}

\subsection{ACC for numerically trivial pairs}

\begin{theorem} \label{t-dc-d} Theorem \ref{t-d} in dimension $n-1$ and Theorem \ref{t-c}
in dimension $n$ implies Theorem \ref{t-d} in dimension $n$.
\end{theorem}
\begin{proof} Let $J_0$ be the finite subset given by applying Theorem \ref{t-d} in
dimension $\leq n-1$ for $J:=D(I)$ and $I_1 \subset I$ the finite subset defined in Lemma
\ref{l-5.2}.

Let $(X,B)$ be an $n$-dimensional projective log canonical pair such that $K_X+B\equiv 0$
and $\coeff (B)\subset I$.  By Theorem \ref{t-model} we may assume that $(X,B)$ is dlt and
in particular $(X,B)$ is klt if and only if $\lfloor B\rfloor =0$.  Let $B=\sum b_iB_i$
where $b_i\in I$.  If $B_i$ intersects a component $S$ of $\lfloor B\rfloor$, then let
$K_S+\Theta =(K_X+B)|_S$.  Note that the coefficients of $\Theta$ belong to the DCC set
$J=D(I)$ (cf. Theorem \ref{tSLA}).  Since $(S,\Theta )$ is log canonical and
$K_S+\Theta \equiv 0$, by Theorem \ref{t-d} in dimension $\leq n-1$, it follows that
$\coeff (\Theta ) \subset J_0$.  If $P$ is an irreducible component of $\Supp(B_i)|_S$,
then
\[
\mult_P(\Theta) =\frac {m-1+f+kb_i}m \qquad \text{for some $f\in J$ and $m$, $k\in \mathbb{N}$.}
\]
By Lemma \ref{l-5.2}, $b_i$ belongs to the finite subset $I_1\subset I$.

We may therefore assume that if $b_i\not \in I_1$, then $B_i\cap \lfloor B\rfloor =0$.
Pick one such component $B_i$ and run the $K_X+B-b_iB_i$ minimal model program with
scaling of an ample divisor.  Since
\[
K_X+B-b_iB_i\equiv -b_iB_i,
\]
every step of this minimal model program is $B_i$ positive and hence does not contract
$B_i$.  Since $K_X+B-b_iB_i\equiv -b_iB_i$ is not pseudo effective, after finitely many
steps we obtain a Mori fibre space
\[
\rmap X.X'.\map .Z..
\]
If at any point we contract a component $S$ of $\lfloor B\rfloor $, then the strict
transforms of $B_i$ and $S$ must intersect and so $b_i\in I_1$ contradicting our
assumptions.  In particular no components of $\lfloor B\rfloor $ are contracted.  If
$\dim Z>0$, then replacing $X$ by a general fibre of $\map X'.Z.$, we see that $b_i$
belongs to $J_0$ (by Theorem \ref{t-d} in dimension $\leq n-1$).  Therefore, we may assume
that $\dim Z=0$ and so $\rho (X')=1$ so that every component of the strict transform of
$\lfloor B\rfloor $ intersects the strict transform of $B_i$.  Arguing as above, if
$\lfloor B\rfloor \ne 0$, it follows that $b_i\in I_1$ which is a contradiction.
Therefore we may assume that $\lfloor B\rfloor =0$, that is, $(X,B)$ is Kawamata log
terminal.  Replacing $(X,B)$ by $(X',B')$ we may also assume that that $\rho(X)=1$.

Let $m=m(n,I)>0$ be the constant whose existence is guaranteed by Theorem \ref{t-c} in
dimension $n$, so that if $(X,C)$ is a projective $n$-dimensional log canonical pair such
that $K_X+C$ is big and $\coeff (C)\subset I$, then $\phi _{m(K_X+C)}$ is birational.  It
suffices to show that $I\cap [(l-1)/m,l/m)$ contains at most one element (for any integer
$1\leq l\leq m$).  Suppose to the contrary that $I\cap [(l-1)/m,l/m)$ contains two
elements say $i_1<i_2$.  We may assume that there is a Kawamata log terminal pair $(X,B)$
as above such that $B=\sum b_jB_j$ where $b_1=i_1$.  Let $\nu \colon\map X'.X.$ be a log
resolution and consider the pair $(X',B':=\nu ^{-1}_*(B+(i_2-i_1)B_1)+\Exc(\nu))$.  Since
$$K_{X'}+B '=\nu ^*(K_X+B)+(i_2-i_1)\nu ^{-1}_*B_1+F$$ where $F\geq 0$ and its support
contains $\Exc(\nu)$, it follows that $K_{X'}+B '$ is big and the coefficients of $B'$ are
in $I$. So by Theorem \ref{t-c}, $\phi _{m(K_{X'}+B')}$ is birational.  In
particular $$\lfloor m(K_{X}+\nu _*B')\rfloor =m(K_{X}+\lfloor m\nu _*B'\rfloor /m)$$ is
big.  Since
\[
(l-1)/m\le i_1<i_2<l/m,
\]
it follows that $\lfloor mi_2\rfloor =l-1$ and so $B\geq \lfloor m\nu _*B'\rfloor /m$.
But since $K_X+B\equiv 0$ this contradicts the bigness of
$K_{X}+\lfloor m\nu _*B'\rfloor /m$.
\end{proof}

\subsection{ACC for the log canonical threshold}

\begin{theorem}\label{t-c-a} Theorems \ref{t-b}, \ref{t-c}, and \ref{t-a} in dimension
$n-1$ imply Theorem \ref{t-a} in dimension $n$.
\end{theorem}
\begin{proof} Since $(X,\Phi )$ is klt, Theorem \ref{t-model} implies that there exists a
small birational morphism $\nu \colon\map X'.X.$ such that $X'$ is $\mathbb{Q}$-factorial.
Since $K_{X'}+\nu ^{-1}_*B=\nu ^*(K_X+B)$, it follows that $(X',\nu ^{-1}_*B)$ is log
canonical if and only if $(X,B)$ is log canonical. Replacing $X$, $\Phi$, and $B$ by $X'$,
$\nu ^{-1}_*\Phi$ and $\nu ^{-1}_*B$, we may assume that $X$ is $\mathbb{Q}$-factorial.

Let $\lambda $ be the log canonical threshold of $(X,B)$ so that $(X,\lambda B)$ is log
canonical but not Kawamata log terminal.  As we are assuming Theorem \ref{t-b} and Theorem
\ref{t-c} in dimension $n-1$, we know that Theorem \ref{t-psef} holds in dimension $n-1$.
Let $\beta <1$ be the constant defined by Theorem \ref{t-psef} in dimension $n-1$ (where
we take $D(I)$ to be the coefficient set).  Let $\mu<1$ be the constant defined by Theorem
\ref{t-a} in dimension $n-1$.  It suffices to show that if $\mu<\lambda <1$, then
$\lambda \leq \beta$.

If $\lambda B$ has a component of coefficient 1, then as
$\coeff (B)\subset I \subset [0,1]$, it follows that $\lambda =1$ and hence $(X,B)$ is log
canonical.  We may therefore assume that all non Kawamata log terminal centres of
$(X,\lambda B)$ have codimension $\geq 2$.  Since $(X,\Phi)$ is klt, by tie breaking
(cf. Proposition \ref{p-mincentre}), there exists a non Kawamata log terminal place $E$ of
$(X,\lambda B)$ and a log canonical pair $(X,\Psi)$ such that $E$ is the unique
non-Kawamata log terminal place of $(X,\Psi )$.  By Theorem \ref{t-model}, there exists a
projective birational morphism $\nu \colon\map X'.X.$ such that $\rho (X'/X)=1$,
$\Exc(\nu )=E$ is an irreducible divisor and
\[
K_{X'}+\lambda \nu ^{-1}_*B+E=\nu ^*(K_X+\lambda B)
\] 
is log canonical so that
\[
K_E+\Diff _{E}(\lambda\nu ^{-1}_* B)=(K_{X'}+\lambda \nu ^{-1}_*B+E)|_E
\] 
is log canonical.  Note that $$K_{X'}+\nu ^{-1}_*\Psi +E=\nu ^*(K_X+\Psi )$$ is plt and
hence $K_E+\Psi _E=( K_{X'}+\nu ^{-1}_*\Psi +E)|_E$ is klt.  Since $\lambda \leq 1$, then
by Lemma \ref{l-diff},
\[
\Diff _{E}(\lambda\nu ^{-1}_* B)\geq \lambda \Diff _E(\nu ^{-1}_*B).
\]  
As $\lambda>\mu$, Theorem \ref{t-a} in dimension $n-1$ implies that
$K_E+ \Diff _E(\nu ^{-1}_*B)$ is log canonical.  Let $H$ be a general sufficiently ample
divisor on $X$, since
\[
K_E+\Diff _E(\nu ^{-1}_*B)\sim_{\mathbb{Q},X}  \Diff _E(\nu ^{-1}_*B)-\Diff _E(\lambda \nu ^{-1}_*B)
\]
is ample over $X$, then
\[
K_E+\Diff _E(\nu ^{-1}_*B)+\nu ^*H|_E
\]
is ample and so by Theorem \ref{t-psef} in dimension $\leq n-1$,
$K_E+t\Diff _E(B)+\nu ^*H|_E$ is big for any $t>\beta$.  In particular since
\[
K_E+\lambda\Diff _{E}(\nu ^{-1}_* B)\leq K_E+\Diff _{E}(\lambda\nu ^{-1}_* B)\equiv _X 0
\] 
it follows that $\lambda \leq \beta$.
\end{proof}

We have the following corollary.
\begin{theorem}[The ACC for the log canonical threshold]\label{t-acc} Fix $n\in \mathbb{N}$ and sets
$I\subset [0,1]$, $J\subset (0,+\infty )$ which satisfy the DCC.  Let $\mathcal L$ be the
set of log canonical thresholds of pairs $(X,B)$ with respect to an $\mathbb{R}$-Cartier
divisor $D$ such that
\begin{enumerate}
\item $(X,B)$ is an $n$-dimensional log canonical pair,
\item $\coeff (B)\subset I$, and 
\item $\coeff (D)\subset J$.
\end{enumerate} 
Then $\mathcal L$ satisfies the ACC.
\end{theorem}
\begin{proof} Replacing $X$ by a $\mathbb{Q}$-factorial modification (cf. Theorem
\ref{t-model}), we may assume that $X$ is $\mathbb{Q}$-factorial.  Suppose that there is a
sequence of triples $(X_i,B_i,D_i)$ as above such that $\lambda _i=\lct (X_i,B_i;D_i)$ is
non-decreasing.  If $\lambda =\lim \lambda _i$ and $K=I+\lambda J$ then $K$ satisfies the
DCC and $(X_i,B_i+\lambda _iD_i)$ is log canonical but not Kawamata log terminal for all
$i=1,2,\ldots$.  We claim that all but finitely many of the coefficients of
$B_i+\lambda D_i$ belong to $[0,1]$.  If this were not the case, then consider a
subsequence such that $\mult_{P_i}(B_i+\lambda D_i)>1$.  We may assume that
$\lambda _i\geq \lambda /2 >0$.  Since $\mult_{P_i}(\lambda _i D_i)\leq 1$, it follows
that $\mult_{P_i}(D_i)\leq 1/\lambda _i \leq 2/\lambda$.  But then
\begin{align*} 
1&\leq \mult_{P_i}(B_i+\lambda D_i)\\
 &=\mult_{P_i}(B_i+\lambda_i D_i)+(\lambda -\lambda _i)\mult_{P_i}( D_i)\\
 &\leq 1+\frac {2(\lambda -\lambda _i)}\lambda.
\end{align*} 

Since $ \mult_{P_i}(B_i+\lambda D_i)$ belongs to the DCC set $K$ and since
$\lim \frac {2(\lambda -\lambda _i)}\lambda=0$, this is a contradiction.  Therefore we may
assume that the coefficients of $B_i+\lambda D_i$ belong to the DCC set $K\cap [0,1]$.

Note that for any $\delta >0$,
\[
(1-\delta )(B_i+\lambda D_i)\leq (B_i+\lambda _i D_i)
\] 
for all $i\gg 0$ and hence $(X_i, (1-\delta )(B_i+\lambda D_i))$ is log canonical.  By
Theorem \ref{t-a}, $(X_i, B_i+\lambda D_i)$ is also log canonical.  But then
$\lambda =\lambda _i$ and hence the sequence $\lambda _i$ is eventually constant as
required.
\end{proof}

\makeatletter
\renewcommand{\thetheorem}{\thesection.\arabic{theorem}}
\@addtoreset{theorem}{section}
\makeatother
\section{Boundedness}

\begin{proposition}\label{c-model} Fix $w\in \mathbb{R} _{>0}$, $n\in \mathbb{N}$ and
a set $I\subset [0,1]$ which satisfies the DCC.  Let $(Z,D)$ be a projective log smooth
$n$-dimensional pair where $D$ is reduced and $\mathbf M _D$ the $b$-divisor corresponding
to the strict transform of $D$ plus the exceptional divisors.  Then there exists
$f\colon\map Z'.Z.$, a finite sequence of blow ups along strata of $\mathbf M_D$, such
that if
\begin{enumerate} 
\item $(X,B)$ is a projective log smooth $n$-dimensional pair
\item $g\colon\map X.Z.$ is a finite sequence of blow ups along strata of $\mathbf M _D$,
\item $\coeff (B)\subset I$,
\item $g_*B\leq D$, and
\item $\vol(X,K_X+B)=w$,
\end{enumerate} 
then $\vol(Z,K_{Z'}+\mathbf M_{B,Z'})=w$ where $\mathbf M_{B,Z'}$ is the strict transform
of $B$ plus the $\rmap Z'.X.$ exceptional divisors.
\end{proposition}
\begin{proof} We may assume that $1\in I$.  Let
\[
V=\{\, \vol(Y,K_Y+\Gamma) \,|\, (Y,\Gamma )\in \mathfrak{D}\,\}
\]
where $\mathfrak{D}$ is the set of all $n$-dimensional projective log smooth pairs such
that $K_Y+\Gamma$ is big, $\coeff (\Gamma )\subset I$, $g\colon\map Y.Z.$ is a birational
morphism and $g_*\Gamma \leq D$.  By Theorem \ref{t-voldcc}, $V$ satisfies the DCC.
Therefore, there is a constant $\delta >0$ such that if
$\vol (Y,K_Y+\Gamma )\leq w+\delta$, then $\vol(Y,K_Y+\Gamma )=w$.  Notice also that by
Theorem \ref{t-psef} there exists an integer $r>0$ such that if
$(Y,\Gamma )\in \mathfrak{D}$, then $K_Y+\frac {r-1}r\Gamma $ is big.  We now fix
$\epsilon >0$ such that
\[
(1-\epsilon )^n>\frac w{w+\delta },\qquad \text{and let}\qquad a=1-\frac \epsilon r.
\]
Since $$K_Y+a\Gamma =(1-\epsilon)(K_Y+\Gamma )+\epsilon (K_Y+\frac {r-1}r\Gamma )$$ it
follows that
\begin{align*} 
\vol(Y,K_Y+a\Gamma ) &\geq \vol ((1-\epsilon)(Y,K_Y+\Gamma))\\
                     &=(1-\epsilon)^n\vol(Y,K_Y+\Gamma)\\
                     &>\frac w{w+\delta }\vol(Y,K_Y+\Gamma).
\end{align*} 

Since $(Z,aD)$ is Kawamata log terminal, there is a sequence of blow ups
$f\colon\map Z'.Z.$ of the strata with the following property: if
$K_{Z'}+\Psi=f^*(K_Z+aD)+E$ where $\Psi \wedge E=0$, then $(Z',\Psi)$ is terminal.  Let
$\mathfrak{F}$ be the set of pairs $(X,B)$ satisfying properties (1-5) above such that
$\phi\colon \rmap X.Z'.$ is a morphism.  If $(X,B)\in \mathfrak{F}$ and $B_{Z'}=\phi _*B$,
then $f_*(aB_{Z'})\leq aD$ so that if
\[
K_{Z'}+\Phi=f^*(K_Z+f_*(aB_{Z'}))+F
\] 
where $\Phi \wedge F=0$, then $(Z,\Phi )$ is terminal.  We then have
\begin{align*} 
\vol(Z',K_{Z'}+aB_{Z'}) &=\vol(Z',K_{Z'}+aB_{Z'}\wedge \Phi )\\
                       &=\vol(X,K_X+\phi ^{-1}_*(aB_{Z'}\wedge \Phi ))\\
                       &\leq\vol(X,K_X+B),
\end{align*} 
where the first line follows from Lemma \ref{l-volumes}(3), the second since
$(Z',aB_{Z'}\wedge \Phi)$ is terminal and the third since
$\phi ^{-1}_*(aB_{Z'}\wedge \Phi )\leq B$.  But then
\[
w\leq \vol (Z',K_{Z'}+B_{Z'})\leq \frac {w+\delta }w\vol (Z',K_{Z'}+aB_{Z'})\leq w+\delta.
\]
By what we observed above, we then have $\vol(Z',K_{Z'}+B_{Z'})=w$ as required.

To conclude the proof, it suffices to observe that if $(X,B)$ is a pair satisfying
properties (1-5) above, then after blowing up $X$ along finitely many strata of
$\mathbf M _D$ and replacing $B$ by its strict transform plus the exceptional divisors, we
may assume that $\rmap X.Z'.$ is a morphism and hence that $(X,B)\in \mathfrak{F}$.
\end{proof}

\begin{proposition}\label{p-lc} Fix $n\in \mathbb{N}$, $d>0$ and
a set $I\subset [0,1]\cap \mathbb{Q}$ which satisfies the DCC.  Let
$\mathfrak{F}_{\rm lc}(n,d,I)$ be the set of pairs $(X,B)$ which are the disjoint union
of log canonical models $(X_i,B_i)$ where $\dim X_i=n$, $\coeff (B_i)\subset I$ and
$(K_X+B)^n=d$.  Then $\mathfrak{F} _{\rm lc}(n,d,I)$ is bounded.
\end{proposition}
\begin{proof} Since $d=\sum d_i$ where $d_i=(K_{X_i}+B_i)^n$ and by Theorem \ref{t-bb} the
$d_i$ belong to a DCC set, it follows easily that there are only finitely many
possibilities for the $d_i$.  We may therefore assume that $X$ is irreducible.  It
suffices to show that there is an integer $N>0$ such that if $(X,B)$ is an $n$-dimensional
log canonical model with $\coeff(B)\subset I$ and $(K_X+B)^n=d$, then $N(K_{X}+B)$ is very
ample.  Suppose that this is not the case and let $(X_i,B_i)$ be a sequence of
$n$-dimensional log canonical models with $\coeff (B_i)\subset I$ and $(K_{X_i}+B_i)^n=d$
such that $i!(K_{X_i}+B_i)$ is not very ample for all $i>0$.

By Theorem \ref{t-bb} and Proposition \ref{p-BB}, the set of such pairs $(X_i,B_i)$ is log
birationally bounded.  Therefore there is a projective morphism $\pi \colon\map Z.T.$ and
a log pair $(Z,D)$ which is log smooth over a variety $T$, such that for any pair
$(X_i,B_i)$ as above, there is a closed point $t_i\in T_i$ and a birational map
$f_i\colon\rmap Z_{t_i}.X_i.$ such that the support of $D_{t_i}$ contains the strict
transform of $B_i$ plus the $f_i$ exceptional divisors.  Passing to a subsequence, we may
assume that the $t_i$ belong to a fixed irreducible component of $T$.  We may therefore
assume that $T$ is irreducible and the components of $D$ are geometrically irreducible
over $T$.

Applying Proposition \ref{c-model} to $(Z_{t_1},D_{t_1})$, we obtain a model
$\map Z_{t_1}'.Z_{t_1}.$ and $\map Z'.Z.$ the morphism obtained by blowing up the
corresponding strata of $\mathbf M _D$.

Denote by $\Phi _{t_i}=({f'}_i^{-1})_*B_i+\Exc(f'_i) \leq D'_{t_i}$, where
$f'_i\colon\rmap Z'_{t_i}.X_i.$ is the induced birational map.  Passing to a subsequence,
we may also assume that for any irreducible component $P$ of the support of
$D':={\mathbf M }_{D,Z'}$, the coefficients of $\Phi _{t_i}$ along $P_{t_i}$ are
non-decreasing.  Let $\Phi^i$ be the divisor with support contained in $D'$ such that
$\Phi^i|_{Z'_{t_i}}=\Phi _{t_i}.$

We claim that for any pair $(X_i,B_i)$ as above
\[
\vol(Z'_{t_i},K_{Z'_{t_i}}+\Phi_{t_i})=d.
\] 
To see this, by the
proof of Corollary of \ref{c-voldcc}, we can construct $\map Z''.Z'.$ by a sequence of
blow ups along strata of $\mathbf M _D$ such that $\rmap Z''_{t_i}.X_i.$ is a rational
map and we have 
\[
\vol(Z''_{t_i}, K_{Z''_{t_i}}+\Psi _{t_i})=d
\]
where $\Psi _{t_i}$ is the strict transform of $B_i$ plus the $Z''_{t_i}/ X_i$ exceptional
divisors.  If $\Psi$ is the divisor supported on $\Supp(\mathbf M _{D,Z''})$ such that
$\Psi |_{Z''_{t_i}}=\Psi _{t_i}$, then
\begin{align*} 
d  &=\vol(Z''_{t_i},K_{Z''_{t_i}}+\Psi _{t_i})\\
   &=\vol(Z''_{t_1},K_{Z''_{t_1}}+\Psi |_{Z''_{t_1}})\\
   &=\vol(Z'_{t_1}, K_{Z'_{t_1}}+\Phi^i |_{Z'_{t_1}})\\
   &=\vol(Z'_{t_i},K_{Z'_{t_i}}+\Phi_{t_i}),
\end{align*} 
where the second and fourth equalities follow from Theorem \ref{t-definv} and the third
one follows from Proposition \ref{c-model}.

By Theorem \ref{t-gmm}, we may assume $(Z',\Phi ^1)$ has a relative log canonical model
$\psi \colon\rmap Z'.W.$ over $ T$, which fibere by fibre
$\psi _{t_i}\colon\rmap Z'_{t_i}.W_{t_i}.$ gives the log canonical model for
$(Z'_{t_i},\Phi ^1_{t_i})$ for all $i\geq 1$.  Notice that by Theorem \ref{t-definv},
\[
d=\vol(Z'_{t_k},K_{Z'_{t_k}}+\Phi ^k_{t_k})=\vol(Z'_{t_1},K_{Z'_{t_1}}+\Phi ^k_{t_1})
\]
for all $k>0$. Since we have assumed that
\[
\Phi ^1\leq \Phi ^2\leq \Phi^3\leq \ldots,
\] 
it follows by Lemma \ref{l-222} that $\psi _{t_i}\colon\rmap Z'_{t_i}.W_{t_i}.$ is also a
log canonical model of $(Z'_{t_i},\Phi ^k_{t_i})$ for all $k\geq 1$,
$\psi _{t_i*}\Phi ^k_{t_i}=\psi _{t_i*}\Phi ^1_{t_i}$ and there is an isomorphism
$\alpha_i\colon W_{t_i}\cong X_i$.
 
There is an integer $N>0$ such that $N(K_W+\psi _*\Phi ^1)$ is very ample over $T$ and so
$N(K_{W_{t_i}}+\psi _{t_i*} \Phi _{t_i}^1)$ is very ample for all $i>0$.  Since
\[
\psi _{t_i*}\Phi _{t_i}^1=\psi _{t_i*}\Phi _{t_i}^i=\alpha_i^*(B_i),
\] 
it follows that $N(K_{X_i}+B_i)$ is very ample for all $i>0$ which is the required
contradiction.
\end{proof}

\begin{lemma}\label{l-222} Let $(X,B)$ be a log canonical pair such that $K_X+B$ is big
and $f\colon\rmap X.W.$ the log canonical model of $(X,B)$.  If $B'\geq B$, $(X, B')$ is
log canonical and $\vol(X,K_X+B)=\vol(X,K_X+B')$, then $f$ is also the log canonical model
of $(X,B')$.
\end{lemma}
\begin{proof} Replacing $X$ by an appropriate resolution, we may assume that
$f\colon\map X.W.$ is a morphism. Let $A=f_*(K_X+B)$, then $A$ is ample and
$F:=K_X+B-f^*A$ is effective and $f$-exceptional.  We have for any $t\geq 0$
\begin{align*} 
\vol(X,K_X+B) &=\vol (X,K_X+B+t(B'-B))\\
              &\geq \vol(X,f^*A+t(B'-B))\\
             & \geq \vol (X,f^*A)\\
             & =\vol (X,K_X+B).
\end{align*} 

But then
\[
\vol(X,f^*A+t(B'-B))=\vol(X,f^*A)\qquad \forall t\in [0,1]
\]
is a constant function.  If $E$ is a component of $B'-B$ then by \cite{LM09} we have
\begin{align*} 
0 &=\frac d{dt}\vol(X,f^*A+tE)|_{t=0}\\
  &=n\cdot \vol_{X|E}(H)\\
  &\geq n \cdot E \cdot f^*A^{n-1}\\
  &=n\cdot \deg f_*E.
\end{align*} 
Therefore $E$ is $f$-exceptional and so 
\begin{align*} 
H^0(X,\ring X.(m(K_X+B'))) &=H^0(X,\ring X.(mf^*A+m(E+F)))\\
                          &=H^0(X,\ring X.(mf^*A))\\ 
                          &=H^0(X,\ring X.(m(K_X+B)))
\end{align*} 
and thus $f$ is the log canonical model of $(X,B')$. 
\end{proof}

\begin{proof}[Proof of \ref{t-bound}] Let $(X,B)\in \mathfrak{F} _{\rm slc}(n,I,d)$ and
$\map X^\nu.X.$ be its normalisation.  By Proposition \ref{p-lc}, if we write
\[
X^\nu=\coprod X_i\qquad \text{and} \qquad (K_X+B)|_{X_i}=K_{X_i}+B_i,
\] 
then the pairs $(X_i,B_i)$ are bounded.  In particular, there exists a finite set of
rational numbers $I_0\subset I$ such that
\[
\coeff (B_i)\subset I_0 \qquad \text{and} \qquad (K_{X_i}+B_i)^{n}=d_i\in \mathfrak{D}.
\]
By \cite{Kollar11} and \cite[5.3]{Kollar13}, the slc models $(X,B)$ are in one to one
correspondence with pairs $(X^\nu,B^\nu)$ and involutions $\tau\colon\map S^\nu.S^\nu.$ of
the normalisation of a divisor $S\subset \lfloor B^\nu \rfloor$ (the divisor $S$
corresponds to the double locus of $\map X^\nu.X.$) such that $\tau$ sends the different
$\Diff _{S^\nu}(B^\nu)$ to itself.  Since $\tau$ is an involution that fixes the ample
$\mathbb{Q}$-divisor $(K_{X^\nu}+B^\nu)|_S$, it follows that $\tau $ belongs to an
algebraic group.  Since fixing the different $\Diff _{S^\nu}(B^\nu)$ is a closed condition
the set of possible involutions $\tau$ corresponds to a closed subset of this algebraic
group and so $\tau$ is bounded.  Therefore the quadruples $(X,B,S,\tau)$ are bounded.
\end{proof}
\bibliographystyle{hamsplain}
\bibliography{/home/mckernan/Jewel/Tex/math}

\end{document}